\documentclass[10pt,a4paper]{article}

\usepackage[utf8]{inputenc}
\usepackage[english]{babel}
\usepackage{amsmath}  
\usepackage{amsfonts} 
\usepackage{amsthm}    
\usepackage{amssymb}
\usepackage{graphicx, overpic} 
\usepackage{paralist} 
\usepackage{geometry} 
\usepackage{xcolor, comment}

\author{Christian Bick${}^{1,2,3,4}$ and Tobias Böhle${}^{1,5}$ and Christian Kuehn${}^{5,6,7}$}
\title{Phase Oscillator Networks with Nonlocal Higher-Order Interactions: Twisted States, Stability and Bifurcations}
\date{{\small
${}^1$Department of Mathematics, Vrije Universiteit Amsterdam, De Boelelaan 1111, Amsterdam, the Netherlands.\\
${}^2$Institute for Advanced Study, Technical University of Munich, Lichtenbergstr. 2, 85748 Garching, Germany.\\
${}^3$Department of Mathematics, University of Exeter, Exeter EX4 4QF, United Kingdom.\\
${}^4$Mathematical Institute, University of Oxford, Oxford OX2 6GG, United Kingdom.\\
${}^5$Technical University of Munich, School of Computation Information and Technology, Department of Mathematics, Boltzmannstr. 3, 85748 Garching, Germany.\\
${}^6$Complexity Science Hub Vienna, Josefst\"adter Str.~39, 1080 Vienna, Austria.\\
${}^7$Munich Data Science Institute, Walther-von-Dyck-Str. 10, 85748 Garching, Germany.\\}
\medskip
\today}

\newcommand{\R}{\mathbb{R}} 
\newcommand{\Z}{\mathbb{Z}} 
\newcommand{\N}{\mathbb{N}}
\newcommand{\T}{\mathbb{T}}
\renewcommand{\S}{\mathbb{S}}

\newcommand{\norm}[1]{\left\lVert #1 \right\rVert}
\newcommand{\abs}[1]{\left\lvert #1 \right\rvert}
\newcommand{\Tobi}[1]{{\color{orange}#1}}
\newcommand{\CB}[1]{{\color{purple}CB: #1}}
\renewcommand{\d}{\mathrm d}

\theoremstyle{definition}
\newtheorem{definition}{Definition}[section]		

\theoremstyle{plain} 	
\newtheorem{lemma}[definition]{Lemma}				
\newtheorem{theorem}[definition]{Theorem}			
\newtheorem{proposition}[definition]{Proposition}	

\theoremstyle{remark}
\newtheorem{remark}[definition]{Remark}				

\theoremstyle{plain}

\numberwithin{equation}{section}

\begin{document}

\maketitle

\hrule
\paragraph{Abstract.}

The Kuramoto model provides a prototypical framework to synchronization phenomena in interacting particle systems. Apart from full phase synchrony where all oscillators behave identically, identical Kuramoto oscillators with ring-like nonlocal coupling can exhibit more elaborate patterns such as uniformly twisted states. It was discovered by Wiley, Strogatz and Girvan in 2006 that the stability of these twisted states depends on the coupling range of each oscillator. In this paper, we analyze twisted states and their bifurcations in the infinite particle limit of ring-like nonlocal coupling. We not only consider traditional pairwise interactions as in the Kuramoto model but also demonstrate the effects of higher-order nonpairwise interactions, which arise naturally in phase reductions. We elucidate how pairwise and nonpairwise interactions affect the stability of the twisted states, compute bifurcating branches, and show that higher-order interactions can stabilize twisted states that are unstable if the coupling is only pairwise.
\vspace{-3mm}
\vspace{-3mm}
 \vspace{4mm}
\hrule

\section{Introduction}

Interacting particle systems are abundant in many real-world systems. For example fireflies flashing in unison~\cite{Buck1968}, collective behavior in the financial market~\cite{DalMasoPeron2011}, chirping snowy tree crickets~\cite{Walker1969}, or neurons in the brain synchronizing their bursts~\cite{Golomb2001} are all examples of dynamical systems that consist of many interacting particles. Even though each particle behaves according to its own microscopic rules, the system as a whole can show collective dynamics, for example synchrony.
A classical mathematical model to study synchrony in these interacting particle systems is the Kuramoto model, that was proposed by Yoshiki Kuramoto in 1984~\cite{Kuramoto1984}. It describes the evolution of~$M$ oscillators on the unit circle, each coupled to every other via a simple sinusoidal coupling. Moreover, each oscillator possesses an intrinsic frequency that is typically sampled from a real probability distribution with a unimodal symmetric density, which makes the oscillators heterogeneous.

While the classical Kuramoto model assumes all-to-all identical coupling, in many real-world systems the coupling is actually not all-to-all but interactions are captured by a (weighted) graph. 
As an example, each oscillator may have a spatial position and coupling between oscillators depends on their positions; such networks often arise in neural field modeling~\cite{Amari1977,Coombes2005,Borisyuk1998}, where coupling strength typically relates to the distance of nodes.
Such a coupling scheme can be realized by considering Kuramoto oscillators on a graph.
Consider~$M$ oscillators and let~$a_{ij}$ be the coupling from oscillator~$j$ to oscillator~$i$; the matrix~$(a_{ij})$ can be interpreted as the adjacency matrix of a weighted graph. The phase~$\phi_i\in\T:= [0,2\pi]/(0\sim2\pi)$ of oscillator~$i$ evolves according to
\begin{align*}
    \dot \phi_i = \omega_i + \frac 1M \sum_{j=1}^M a_{ij} \sin(\phi_j-\phi_i),
\end{align*}
for $i=1,\dots,M$. Graphs that do not describe all-to-all coupling allow for more interesting dynamics beyond full phase synchrony~\cite{Rodrigues2016}. 
For example, one can consider a $k$-nearest-neighbor networks of~$M$ nodes: 
Oscillators are arranged as a ring in ascending order around the unit circle and each oscillator is coupled to all of its~$k$ predecessors and all of its~$k$ successors (modulo~$M$). 
In other words, $a_{ij} = 1$ if $\min(\abs{i-j}, M-\abs{i-j}) \le k$ and~$a_{ij} = 0$ otherwise.
Changing~$k$ changes the coupling range: For $k=1$ we have a ring with local nearest-neighbor coupling, for $k=M/2$ the network is globally all-to-all coupled, and for intermediate~$k$ the coupling is often called nonlocal.
On this network, the Kuramoto model with~$\omega_i\equiv 0$ shows many interesting states. For example regular twisted states~\cite{Wiley2006} or irregular chimera states~\cite{OmelChenko2018, Xie2014} when we allow for phase lag parameters in the coupling function.

While the traditional Kuramoto model assumes interactions between pairs of oscillators, higher-order interactions can have a profound impact on the dynamics; cf.~\cite{Battiston2021, Bick2021b}. Such nonpairwise interactions arise naturally in phase oscillator networks that originate from (higher order) phase reductions and become important for the dynamics as the coupling strength is increased~\cite{Ashwin2016a,Leon2019a}.
Moreover, nonpairwise interactions also arise in rings of nonlocally coupled oscillators: In~\cite{Matheny2019}, the authors consider a network of eight nanoelectromechanical oscillators coupled via higher-order nearest-neighbor interactions. They found that this system exhibits complex and exotic states even though the coupling functions are fairly simple.

It turns out that instead of analyzing twisted states on large finite networks, it is easier to consider them in the continuum limit on the limiting object, also called graphon. Therefore, we consider the dynamics of the continuum limit of large $k$-nearest-neighbor networks, when $k$ grows with the system size~$M$. In particular, the limiting network can be obtained by fixing the coupling range~$r= k/M$ and sending $M\to\infty$. On this continuum limit we analyze the stability and the bifurcation around twisted states. While twisted states have originally been studied in~\cite{Wiley2006}, a lot of research has been done to generalize these results~\cite{Medvedev2017, Medvedev2015, Medvedev2014, Omelchenko2014, Chiba2018, Girnyk2012}.

In our paper, we propose an extension of the pairwise coupling in the continuum limit to higher-order interactions. We study the stability of twisted states and show how this property is influenced by higher-order interactions. To this end, we analyze the bifurcation point where a twisted state looses or gains stability. We investigate which nontrivial equilibria bifurcate from the twisted states and how they depend on the parameters of the system, which are the coupling range~$r$ and the strengths of the higher-order interactions. We apply this theory to regular models without higher-order interactions and thereby extend the analysis from many previous works by the bifurcation analysis. Moreover, we show how higher-order interactions can make $q$-twisted states stable or unstable and how they affect the type of the bifurcation.

The work is organized as follows: In Section~\ref{sec:setting} we introduce the system and clarify general notation. Then, in Section~\ref{sec:solution}, we first perform a Lyapunov--Schmidt reduction to convert the infinite-dimensional problem into a finite-dimensional problem. Then, we explain how to use the symmetry of the system to simplify the finite-dimensional problem into a two-dimensional one. Next, we tackle this two-dimensional problem by employing a Taylor expansion from which we can read off the type of the bifurcation. In the last part of this section, we derive equilibria approximations and analyze linear stability of bifurcating equilibria. Section~\ref{sec:applications} contains three interesting special cases, for which we conduct numerical simulations that illustrate and confirm the theory. In Section~\ref{sec:honns} we discuss a few more ways of generalizing pairwise $k$-nearest-neighbor coupling to higher-order interactions and explain why we focused particularly on one choice. Finally, Section~\ref{sec:conclusion} contains some concluding remarks.

\section{Nonlocally Coupled Phase Oscillators with Higher-Order Interactions}
\label{sec:setting}

\subsection{Nonlocally Coupled Phase Oscillators and $q$-Twisted States}\label{sec:continuumLimit}

Consider a nonlocally coupled network of~$M$ identical Kuramoto phase oscillators with $k$-nearest-neighbor network. More specifically, suppose that the phase $\phi_i$ of oscillator~$i$ evolves according to
\begin{align}\label{eq:classical_finite}
	\dot\phi_i = \frac 1M \sum_{j=1}^M a_{i-j} \sin(\phi_j-\phi_i),\quad \text{for }i=1,\dots,M,
\end{align}
where the the coefficients~$a_i$ are defined by $a_i = 1$ if $\min(\abs{i}, M-\abs{i}) \le k$ and $a_i=0$ otherwise. This ODE system is posed in the phase space~$\T^M$ and the underlying coupling structure is illustrated in Figure~\ref{fig:network_and_twisted}(a).
The system~\eqref{eq:classical_finite} has multiple symmetries. First, there is continuous symmetry of $\phi_i\mapsto\phi_i+\alpha$ for any $\alpha\in\T$. Moreover, since~$a_{i-j}$ only depends on the difference $i-j$ and the network is symmetric, the system also has a finite symmetry group $\mathbb D_M$, which is the dihedral group consisting of~$2M$ elements. For a full investigation of symmetry in this system see~\cite{Ashwin1992}. For $q\in \N$, a $q$-twisted state is a phase configuration that satisfies
\begin{align*}
    \phi^q_i = 2\pi q i/M+\alpha,\quad \text{for }i=1,\dots,M,
\end{align*}
where $\alpha\in \T$ is an arbitrary parameter. 
By exploiting these symmetries and the odd symmetry of the coupling function~$\sin$, one can show that these $q$-twisted states are equilibria of~\eqref{eq:classical_finite}. While showing the time invariance of a $q$-twisted state on~\eqref{eq:classical_finite} is relatively easy, investigating its stability turns out to be more complicated~\cite{Girnyk2012}. 

To understand the stability of $q$-twisted states one often considers the continuum limit of the network dynamical system with infinitely many oscillators.
Given a solution $\phi_i^M(t)$ of the system~\eqref{eq:classical_finite}, one can derive the continuum limit by first defining a function~$\Theta^M(t,x)$ as
\begin{align*}
    \Theta^M(t,x) = \phi_i(t)\qquad \text{if }x\in \left[\frac{i-1}M, \frac iM\right).
\end{align*}
Here, $x$~represents the position of an oscillator in the infinite network limit. In order to distinguish the phase space~$\T$ from the index set of the network nodes, we regard~$x$ as a variable on the unit circle $\S := [0,1]/(0\sim 1)$, that we parameterize from~$0$ to~$1$. 
Then, the function~$\Theta^M(t,x)$ satisfies
\begin{align}\label{eq:continuum_finite}
    \frac{\partial}{\partial t}\Theta^M(t,x) = \int_\S W_r^M(x-y) \sin(\Theta^M(t,y)-\Theta^M(t,x))\ \d y,
\end{align}
where~$W_r^M$ is defined as
\begin{align*}
    W_r^M(x) = a_i\qquad \text{if } i \in \left[\frac{i-1}M,\frac iM\right),
\end{align*}
for $x\in \S$. Keeping the coupling range $r:=k/M \in (0,\frac 12]$ fixed and letting $M\to \infty$, we formally obtain the limit $W_r^M\to W_r \in L^2(\S)$, where
\begin{align}\label{eq:W_def}
	W_r(x) := \begin{cases} 1 & \quad \text{if } \min(x, 1-x)\le r\\ 0 &\quad \text{else}\end{cases}.
\end{align}
Now, suppose that $\lim_{M\to \infty} \Theta^M(t,x) = \Theta(t,x)$ for a function $\Theta$. Then, formally taking the limit of~\eqref{eq:continuum_finite} as $M\to \infty$, we obtain the continuum limit
\begin{align}\label{eq:continuum_graph}
    \frac{\partial}{\partial t}\Theta(t,x) = \int_\S W_r(x-y) \sin(\Theta(t,y)-\Theta(t,x))\ \d y.
\end{align}
In this continuum limit, a $q$-twisted state, see Figure~\ref{fig:network_and_twisted}(c), is given by
\begin{align*}
    \Theta^q(x) = 2\pi q x + \alpha.    
\end{align*}
Even though this derivation was formal, it can be rigorously shown that the continuum limit~\eqref{eq:continuum_graph} approximates the dynamics of the finite system~\eqref{eq:classical_finite} for large~$M$~\cite{Medvedev2014a,Medvedev4,KuehnThrom2,GkogkasKuehnXu}.

System~\eqref{eq:continuum_graph} has a continuous $\T\times\S$ symmetry. A symmetry element $\beta\in\T$ acts 
by a phase shift
\begin{align}\label{eq:symmetry_phase_differences}
    \beta:\Theta(t,x) \mapsto \Theta(t,x)+\beta
\end{align}
to all oscillators and an element~$\phi\in\S$ acts by rotating the ring-like network, that is,
\begin{align}\label{eq:symmetry_rotation}
    \phi:\Theta(t,x)\mapsto \Theta(t,x+\phi).
\end{align}
The latter symmetry action can be seen as part of the limit of the~$\mathbb{D}_M$ symmetry of the finite-dimensional system as $M\to\infty$; cf.~\cite{Bick2021}.

\begin{figure}[h]
\centering
\begin{overpic}[grid=off, width = \textwidth]{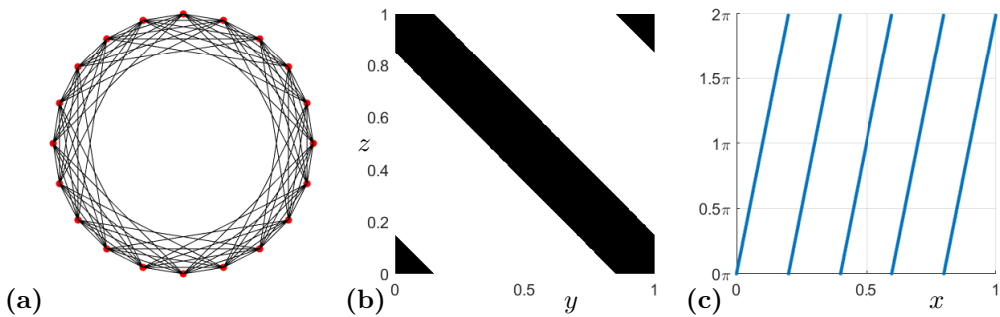}
\put(9,0){\textbf{(a)}}
\put(37,0){\textbf{(b)}}
\put(65,0){\textbf{(c)}}
\put(55,0){$y$}
\put(38,13){$z$}
\put(85,0){$x$}
\end{overpic}
\caption{\textbf{(a)} shows a $5$-nearest-neighbor graph on $20$ nodes. \textbf{(b)} depicts the topology of a higher-order network. All pairs $(y,z)$ for which a triangle between $(x,y,z)$ for $x=0$ exists, i.e., that satisfy $W_r(z+y-2\cdot 0)=1$ are depicted by the black area. Here, $r = 0.15$. \textbf{(c)} shows a $5$-twisted state in the continuum limit.}
\label{fig:network_and_twisted}
\end{figure}

\subsection{Nonlocal Coupling and Higher-Order Interactions}

Higher-order interactions for phase oscillators are nonlinear interactions between three or more oscillator phases. An example for triplet interactions between three oscillators at $x,y,z\in\S$ with phases $\Theta(t,x),\Theta(t,y),\Theta(t,z)$ at time~$t$ is given by $\sin(\Theta(t,z)+\Theta(t,y)-2\Theta(t,x))$.
To incorporate nonpairwise higher-order interactions, we extend the continuum limit~\eqref{eq:continuum_finite} in the natural way: For a network with pure triplet interactions the phases evolve according to
\begin{align*}
    \frac{\partial}{\partial t}\Theta(t,x) = \int_\S \int_\S W(z,y,x) \sin(\Theta(t,z)+\Theta(t,y)-2\Theta(t,x))\ \d y \d z,
\end{align*}
where $W\colon \S^3\to \R$ is a general $3$-tensor that describes in which triplets there is an interaction (and generalizes the weighted adjacency tensor for pairwise interactions).

As we are interested in $q$-twisted states, we consider a specific class of nonlocal higher-order interaction structure $W(z,y,x) = W_r(z+y-2x)$ for the coupling function $\sin(\Theta(t,z)+\Theta(t,y)-2\Theta(t,x))$. This choice of network topology naturally extends nonlocal pairwise interactions and in the resulting network dynamical system $q$-twisted states are still relative equilibria---the resulting network topology can be seen in Figure~\ref{fig:network_and_twisted}(b). While we focus on this particular generalization, there are other network nonlocal higher-order network topologies that preserve $q$-twisted states as discussed in Section~\ref{sec:honns} for a discussion.
The higher-order network topology allows for long range connections: An oscillator~$x$ is not only influenced by triangles spanned by nodes neighboring~$x$. Instead, if say $x=0$, a triangle $(x,y,z)$ exists when $\abs{z+y}\le r$ in $\S$, which is the case when for example $z = 1/4$ and $y = 3/4$. In particular, in this case $z+y-2x=0$ and thus this triangle exists for all coupling ranges $r>0$. In that sense, it is distinct from the ``nearest-neighbor'' higher-order networks considered in~\cite{Matheny2019}.

Our nonlocal higher-order coupling also naturally extends beyond triplet coupling: In the following we will consider a combination of pairwise and higher-order (triplet and quadruplet) interactions. 
Specifically, we consider the continuum limit
\begin{align}\label{eq:problem}
\begin{split}
	\frac{\partial}{\partial t}\Theta(t,x) &= \int_\S W_r(x-y) \sin(\Theta(t,y)-\Theta(t,x))\ \d y\\
	&\quad + \lambda \int_\S\int_\S W_r(z+y-2x) \sin(\Theta(t,z)+\Theta(t,y) - 2\Theta(t,x))\ \d y\d z\\
	&\quad + \mu\int_\S\int_\S\int_\S W_r(z-y+w-x) \sin(\Theta(t,z)-\Theta(t,y)+\Theta(t,w)-\Theta(t,x))\ \d w \d y \d z,
\end{split}
\end{align}
where $r \in (0,\frac 12]$ describes the coupling range and $\lambda,\mu \in \R$ are the strengths of the higher-order interactions. These three parameters can be summarized into one parameter $p = (r,\lambda,\mu) \in \mathcal P$, where $\mathcal P = (0,\frac12]\times\R\times \R$ denotes the parameter space. Here, the first line of~\eqref{eq:problem} describes the continuum limit of a Kuramoto model with nonlocal coupling and the second and third lines are triplet and quadruplet higher-order interactions.
The system~\eqref{eq:problem} has the same symmetries as the system~\eqref{eq:continuum_graph}.

\subsection{Linearization}

We want to analyze the stability and bifurcations of $q$-twisted states for~\eqref{eq:problem} and determine the existence and stability of possible bifurcating branches that occur as system parameters are varied. To answer these kind of questions, eigenvalues of the linearization of the right-hand side of~\eqref{eq:problem} are of importance. 
The continuous phase shift symmetry~\eqref{eq:symmetry_phase_differences} of system~\eqref{eq:problem} implies that if~$\Theta(t,x)$ is a solution to the PDE~\eqref{eq:problem}, then so is $\Theta(t,x) + \beta$ with a constant $\beta\in \T$, see~\cite{Golubitsky2002}.
Therefore, the system has a neutrally stable direction, which yields a zero eigenvalue in the linearization of the right-hand side of~\eqref{eq:problem}.
We can avoid the zero eigenvalue by considering the evolution of phase differences $\Psi(t,x) := \Theta(t,x) - \Theta(t,0)$ which reduces the continuous phase shift symmetry~\eqref{eq:symmetry_phase_differences}. 
The function~$\Psi(t,x)$ satisfies
\begin{align}\label{eq:pd_problem}
\begin{split}
\frac{\partial}{\partial t}\Psi(t,x) &= \int_\S W_r(x-y)\sin(\Psi(t,y)-\Psi(t,x))\ \d y - \int_\S W_r(y)\sin(\Psi(y))\ \d y\\
&\quad + \lambda\left[\int_\S\int_\S W_r(z+y-2x)\sin(\Psi(t,z)+\Psi(t,y)-2\Psi(t,x)) \ \d y \d z\right.\\
&\qquad  \left.- \int_\S\int_\S W_r(z+y)\sin(\Psi(t,z)+\Psi(t,y))\ \d y \d z\right]\\
&\quad + \mu \left[ \int_\S\int_\S\int_\S W_r(z-y+w-x) \sin(\Psi(t,z)-\Psi(t,y)+\Psi(t,w)-\Psi(t,x)) \ \d w \d y \d z\right.\\
&\qquad \left. - \int_\S\int_\S\int_\S W_r(z-y+w) \sin(\Psi(t,z)-\Psi(t,y)+\Psi(t,w)) \ \d w \d y \d z \right]
\end{split}
\end{align}
and $\Psi(t,0) = 0$. We denote the right-hand side of this system by~$F(\Psi,p)$. In this system, a $q$-twisted state is given by $\Psi^q(x) = 2\pi q x$ and it cannot be perturbed along a constant function anymore, since the perturbed function would then violate $\Psi(t,0) = 0$.
Moreover, since we are particularly interested in the behavior of~$F$ in a neighborhood of a $q$-twisted state, we define
\begin{align*}
    F^q(v,p) := F(\Psi^q + v, p),\qquad F^q\colon X\times\mathcal P\to X.
\end{align*}
Here,~$v$ can be seen as the perturbation of a twisted state which we consider in the space~$X := H_0^1:=H_0^1(\S,\R)$ which is the function space whose functions and their weak derivatives are in~$L^2(\S,\R)$ and which satisfy the boundary condition $f(0) = 0$. Since $C(\S)\subset H^1(\S,\R)$, these boundary conditions can be imposed in the classical sense. That~$F^q$ indeed maps into~$X$ is shown as a special case in Appendix~\ref{sec:derivative}. Together with the scalar product
\begin{align*}
f\cdot g := \int_\S f(x)g(x)\ \d x + \int_\S Df(x) Dg(x)\ \d x
\end{align*}
the space~$X$ forms a Hilbert space. Moreover, the induced norm is given by
\begin{align*}
\norm{f}_{H^1_0} = \sqrt{f\cdot f} = \sqrt{\norm{f}_{L^2}^2 + \norm{Df}_{L^2}^2}.
\end{align*}
Since every function $\eta\in H_0^1$ can be written as
\begin{align*}
	\eta(x) = \sum_{k=1}^\infty a_k \sin(2\pi k x) + b_k (1-\cos(2\pi k x)),
\end{align*}
the functions $u_k(x) = \sin(2\pi k x)$ and $w_k(x) = 1-\cos(2\pi k x)$ for $k\ge 1$ form a Schauder basis of~$H_0^1$. Furthermore, our choice of~$W_r$ in~\eqref{eq:W_def} yields a Fourier decomposition
\begin{align*}
	W_r(x) = \frac 12 \hat W_r(0) + \sum_{k=1}^\infty \hat W_r(k) \cos(2\pi k x),
\end{align*}
with
\begin{align}\label{eq:W_hat_coefficients}
	\hat W_r(k) = \begin{cases} \frac{2\sin(2\pi k r)}{\pi k} & \quad \text{if }k\neq 0\\ 4r &\quad\text{if }k=0 \end{cases}.
\end{align}
It can be shown, see Appendix~\ref{sec:derivative}, that the Fr\'echet-Derivative of~$F^q(v,p)$ with respect to~$v$ around~$0$ is given by a bounded linear operator $F^q_v(0,p)\colon X\to X$. An evaluation of~$F^q_v(0,p)$ on these basis elements yields
\begin{align}\label{eq:F_ev}
\begin{split}
	F^q_v(0,p)[u_k] &= c_1(q,k,p) u_k,\\
	F^q_v(0,p)[w_k] &= c_1(q,k,p) w_k,
\end{split}
\end{align}
for $k\in \N$ and
\begin{align*}
	c_1(q,k,p) := \frac{1}{4}\Big(\hat W_r(q-k)+ \hat W_r(q+k) - 2\hat W_r(q) - (4\lambda+2\mu) \hat W_r(q)\Big),
\end{align*}
where we use the convention $\hat W_r(-k) := \hat W_r(k)$. The eigenvalues are then given by $\xi_k = c_1(q,k,p)$, each with multiplicity~$2$. Since $F^q_v(0,p)$ is a multiplication operator on this basis, the spectrum is the closure of the set of eigenvalues, i.e.,
\begin{align*}
	\sigma(F^q_v(0,p)) = \operatorname{cl}(\{ \xi_k, k\in\N\}).
\end{align*}

When one of these spectral values passes through $0$, we may expect a change of stability of the $q$-twisted state. This is what we investigate in the next section.

\section{Bifurcation Theory}\label{sec:solution}

We now consider varying parameters along a general curve in parameter space; for phase oscillator networks that arise as phase reductions from a physical system~\cite{Ashwin2016a}, we expect that variation of a physical system parameter gives rise to such a curve. Specifically, we assume from now on that
\begin{enumerate}
	\item there is a smooth curve through the interior of the parameter space $p\colon (-\delta,\delta)\to \operatorname{int}(\mathcal P)$, $p(s) = (r(s), \lambda(s), \mu(s))$ with $p(0) = p_0 = (r_0, \lambda_0,\mu_0)$,
	\item at $s=0$ we have $c_1(q,\ell,p(s)) = 0$ for one $\ell\in \N$ and for all $s\in(-\delta,\delta)$ we have $c_1(q,k,p(s))\notin (-\epsilon,\epsilon)$ for some $\epsilon >0$ and all $k\neq \ell$,
	\item $c_1(q,\ell,p(s))$ is an isolated eigenvalue, i.e., for all $s\in(-\delta,\delta)$, the sequence $(c_1(q,k,p(s)))_{k\in \N}$ does not have an accumulation point at $\xi_\ell$,
	\item the zero eigenvalue passes through $0$ with non-vanishing speed as $s$ passes through $0$, i.e., $\frac{\d }{\d s} c_1(q,\ell,p(s)) \neq 0 $ for $s=0$.
\end{enumerate}

\begin{remark}
	Since $\lim_{k\to \infty} c_1(q,k,p) = \frac 14 \hat W_r(q)(-2 - (4\lambda+2\mu))$ exists, it is the only possible spectral value that is not an eigenvalue. Moreover, since it depends continuously on all parameters, Assumption~$3$ only has to be checked at $s=0$. Furthermore, since $c_1(q,k,p)$ is uniformly (w.r.t.~$k$) Lipschitz continuous in~$p$, Assumption~$2$ only has to be checked at $s=0$.
\end{remark}

From now on we use the notation~$V$ for an open neighborhood of~$p_0$ in~$\mathcal P$. By a slight abuse of notation, this $V$ might have to be shrunk from one statement to the other, but always represents a small enough open neighborhood of~$p_0$. Similarly, $(-\delta,\delta)$, which represents an open neighborhood of $0$ in $\R$, might have to be shrunk from statement to statement.

\subsection{Lyapunov--Schmidt Reduction}\label{sec:LyapunovSchmidt}

At the bifurcation point $s=0$, the nullspace of the linearization is given by
\begin{align*}
	N := \mathcal N(F^q_v(0,p_0)) = \operatorname{span}\{u_\ell, w_\ell\}.
\end{align*}
Further we denote the range of the linearization by $R=\operatorname{span}\{u_k,w_k:~k\neq \ell\}$. Following the notation from~\cite{Kielhofer2012}, we consider $F^q_v(v,p)$ as a map from $X\times\mathcal P$ to~$Z$, where $X=Z=H_0^1$. Even though~$X=Z$, we use different notation for the domain and target set to emphasize the distinction between them. These spaces can be decomposed into
\begin{align*}
	X = N\oplus X_0\qquad\text{and}\qquad Z = R \oplus Z_0,
\end{align*}
where $X_0$ is a complement of~$N$ in~$X$ and~$Z_0$ is a complement of~$R$ in~$Z$. We choose $Z_0=N$ and $X_0=R$. Moreover the projection onto~$Z_0$ is defined by
\begin{align*}
	&Q\colon Z\to Z_0\quad\text{along }R.
\end{align*}
To determine equilibria of~\eqref{eq:pd_problem}, we need to find solutions to $F^q(v,p)=0$. By performing a Lyapunov-Schmidt reduction we can reduce this infinite-dimensional problem to a finite-dimensional problem, as the next theorem shows.
\begin{theorem}[{\cite[Chapter I.2]{Kielhofer2012}}]\label{thm:LyapunovSchmidt}
    There is a neighborhood $U_1\times V_1\subset X\times\mathcal P$ of~$(0,p_0)$ such that the full infinite-dimensional problem of finding equilibria of~\eqref{eq:pd_problem}, i.e., solving 
\begin{align*}
	F^q(v,p) = 0
\end{align*}
in $U_1\times V_1$ is equivalent to solving
\begin{align*}
	\Phi(v,p) = 0,
\end{align*}
where $\Phi\colon U_2\times V_2\to Z_0$ for $(v,p)\in U_2\times V_2\subset N\times \mathcal P$. Here, $\Phi(0,p_0) = 0$ and $\Phi$ is defined by
\begin{align}\label{eq:Phi_def}
	\Phi(v,p) = QF^q(v + \psi(v,p),p),
\end{align}
where $\psi\colon N\times V_2\to X_0$ is a function satisfying $\psi(0,p_0) = 0$. It is implicitly defined to be the unique solution to the equation
\begin{align}\label{eq:psi_def}
	(I-Q)F^q(v + \psi(v,p),p) = 0
\end{align}
in a neighborhood of $(0,p_0)$.
\end{theorem}

The proof of this theorem relies on the implicit function theorem. For details see~\cite[Chapter I.2]{Kielhofer2012} and for an introduction see~\cite{KuehnBook1}.

We introduce a coordinate representation of the function~$\Phi$ by considering the basis of the dual space~$Z_0'$ of~$Z_0$, which is given by the two functionals~$z_1^*, z_2^*$ with
\begin{align*}
	\langle z_1^*, v \rangle &= 2 \int_\S \sin(2\pi \ell x) v(x) \ \d x,\\
	\langle z_2^*, v \rangle &= -2 \int_\S \cos(2\pi \ell x) v(x) \ \d x,
\end{align*}
where $\langle\cdot,\cdot\rangle$ denotes the dual pairing. Then, we define a function $\hat \Phi\colon U_3\times V_3\to \R^2$, where $U_3\subset \R^2, V_3\subset \R$ are sufficiently small neighborhoods around the origin, as
\begin{align}\label{eq:hatPhi_def}
	\hat\Phi\left(\begin{pmatrix}
	a\\b\end{pmatrix},s\right) := \begin{pmatrix}
	\langle z_1^*, \Phi(a u_\ell + b w_\ell, p(s))\rangle\\
	\langle z_2^*, \Phi(a u_\ell + b w_\ell, p(s))\rangle
	\end{pmatrix}.
\end{align}
Given $(a_0,b_0)^\top$ in a neighborhood of $(0,0)^\top$ and $s\in (-\delta,\delta)$ such that $\hat \Phi((a_0,b_0)^\top,s)=(0,0)^\top$ we then know that $F(\Psi^q + a_0u_\ell + b_0w_\ell + \psi(a_0u_\ell+b_0w_\ell,p(s)),p(s))=0$. Therefore, this first argument of~$F$ represents an equilibrium. Conversely, every equilibrium in a neighborhood of the bifurcation point can be found in that way.

\subsection{Problem Reduction Using Symmetry}\label{sec:symmetry}

While reducing the system to phase differences~\eqref{eq:pd_problem} has reduced the phase-shift symmetry, the symmetry still has the residual rotational symmetry~\eqref{eq:symmetry_rotation}. One can expect that this symmetry is reflected in the bifurcation behavior. Indeed, the Lyapunov--Schmidt reduction can be carried out such that it preserves symmetries; cf.~\cite{Golubitsky2002}. Here, we show explicitly that the reduced equation~\eqref{eq:hatPhi_def} retains the rotational symmetry. This simplifies the system to a one-dimensional problem by eliminating the symmetry.

Specifically, in phase differences, the rotational symmetry~\eqref{eq:symmetry_rotation} acts as an operator $B_\phi\colon X\to X$ for $\phi\in \S$ given by
\begin{align*}
	(B_\phi f)(x) = f(x+\phi) - f(\phi).
\end{align*}
and the right-hand side~$F^q$ is equivariant with respect to this operation. The nullspace~$N$ is spanned by~$u_\ell$ and~$w_\ell$, which can be obtained from each other by shifting one function around the circle and adding a constant such that it satisfies the boundary conditions, i.e., by applying the operator~$B_\phi$. For the reduced equation determined by~$\hat \Phi$, this corresponds to a rotation. Specifically, with
\begin{align*}
	A_\phi = \begin{pmatrix}
	\cos(2\pi \ell \phi)& \sin(2\pi \ell \phi)\\
	-\sin(2\pi \ell\phi)& \cos(2\pi \ell\phi)
	\end{pmatrix}
\end{align*}
for a two-dimensional rotation matrix, we now show that~$\hat \Phi$ is $\S$-equivariant with respect to the action given by~$A_\phi$.

\begin{proposition}\label{prop:hatPhi_symm}
In a neighborhood of the origin, $\hat \Phi$ satisfies
\begin{align}\label{eq:hatPhi_symmetry}
	\hat \Phi \left(A_\phi \begin{pmatrix}	a\\ b	\end{pmatrix}, s\right) = A_\phi \hat \Phi\left(\begin{pmatrix} a\\b \end{pmatrix}, s\right),
\end{align}
for all~$\phi$.
\end{proposition}

\begin{proof}
A straight-forward calculation confirms that $F^q$ satisfies
\begin{align}\label{eq:F_symmetry}
	F^q(B_\phi\eta,p) = B_\phi F^q(\eta,p)
\end{align}
for all~$\eta\in X$ and all~$p\in \mathcal P$. Now, let us see how this property propagates to the function $\psi$. By definition, $\psi$ solves
\begin{align*}
	(I-Q)F^q(v + \psi(v,p),p) = 0
\end{align*}
for all $v\in N, p\in \mathcal P$. Now, choose $v = B_\phi u \in N$ for some $u\in N$ and note that $B_\phi$ leaves $N$ invariant and further commutes with~$Q$. Then, on the one hand
\begin{align}\label{eq:psi_symm_cond1}
	(I-Q) F^q(B_\phi u + \psi(B_\phi u ,p),p)=0.
\end{align}
On the other hand
\begin{align}
\nonumber
0&= B_\phi 0\\
\nonumber
&= B_\phi (I-Q) F^q(u + \psi(u,p),p)\\
\nonumber
&= (I-Q) B_\phi F^q(u + \psi(u,p),p)\\
\nonumber
&= (I-Q) F^q(B_\phi[u+\psi(u,p)], p)\\
\label{eq:psi_symm_cond2}
&= (I-Q) F^q(B_\phi u + B_\phi \psi(u,p), p),
\end{align}
by the symmetry property~\eqref{eq:F_symmetry}. By comparing~\eqref{eq:psi_symm_cond1} and~\eqref{eq:psi_symm_cond2} one sees that
\begin{align}\label{eq:psi_symmetry}
	\psi(B_\phi u,p) = B_\phi\psi(u,p)
\end{align}
for all~$u\in N$ due to the uniqueness of~$\psi$.
Furthermore, the definition~\eqref{eq:Phi_def} of~$\Phi$ implies that for all~$v\in N$
\begin{align}
\nonumber
\Phi(B_\phi v,p) &= QF^q(B_\phi v + \psi(B_\phi v, p), p)\\
\nonumber
&=QF^q(B_\phi v + B_\phi\psi(v,p),p)\\
\nonumber
&=QF^q(B_\phi[v+\psi(v,p)],p)\\
\nonumber
&=QB_\phi F^q(v + \psi(v,p),p)\\
\nonumber
&=B_\phi QF^q(v + \psi(v,p),p)\\
\label{eq:Phi_symmetry}
&=B_\phi \Phi(v,p),
\end{align}
where we have used~\eqref{eq:psi_symmetry} and~\eqref{eq:F_symmetry}.

Since $\Phi\colon N\times V\to N$, where~$V$ is neighborhood of~$p_0$ in $\mathcal P$, for small enough $\abs{a}, \abs{b}$, we can write
\begin{align*}
	\Phi(au_\ell + b w_\ell, p) = c u_\ell + d w_\ell
\end{align*}
for each fixed $p\in V$ and some $c,d\in \R$. Now, by applying $B_\phi$ to both sides of the equation and using the symmetry property~\eqref{eq:Phi_symmetry}, a straight-forward calculation confirms that
\begin{align*}
	\Phi([a\cos(2\pi \ell\phi) + b \sin(2\pi \ell \phi)] u_\ell + [-a\sin(2\pi \ell \phi) + b \cos(2\pi \ell \phi)]w_\ell,p) \\= [c\cos(2\pi \ell\phi) + d \sin(2\pi \ell \phi)] u_\ell + [-c\sin(2\pi \ell \phi) + d \cos(2\pi \ell \phi)]w_\ell.
\end{align*}
Using~\eqref{eq:hatPhi_def}, this yields the result.
\end{proof}

Next, we show that $\hat\Phi$ does not change the angle of a vector but only multiplies its length by a (possibly negative) factor.

To achieve this, we first define the space of odd functions $O$:
\begin{align*}
	O := \{ f \in X: f(x) = -f(-x)\}.
\end{align*}

\begin{lemma}\label{lem:F_odd}
Let $p\in \mathcal P$ and $v\in O$. Then $F^q(v,p)\in O$.
\end{lemma}
This lemma follows by a calculation using linear substitutions of the integrating variables that appear in the definition of $F^q$ and $F$.

Given this lemma, we can consider the restriction of $F^q$ to the space of odd functions $O$:
\begin{align*}
	F^{q,\dagger}\colon X^\dagger\times \mathcal P\to Z^\dagger, \qquad F^{q,\dagger}(v,p) := F^q(v,p),
\end{align*}
where both $X^\dagger = Z^\dagger = O$. We use the symbol~$\dagger$ whenever we are referring to a function or a space that is reduced to~$O$.

Under this restriction $F^{q,\dagger}$ inherits smoothness from $F$ and $F^q$. Following the notation from Section~\ref{sec:LyapunovSchmidt} we denote $N^\dagger := \mathcal N( F^{q,\dagger}_v(0,p_0)) = \operatorname{span}\{u_\ell\}$. Moreover, there are decompositions
\begin{align*}
	X^\dagger = N^\dagger \oplus X_0^\dagger\qquad \text{and} \qquad Z^\dagger = R^\dagger \oplus Z_0^\dagger,
\end{align*}
where $R^\dagger$ is the range of $F^{q,\dagger}_v(0,p_0)$ and we choose $X^\dagger_0 = R^\dagger$ and $Z^\dagger_0 = N^\dagger$. Additionally, we denote $Q^\dagger$ for the restricted projection of $Q$ from $X^\dagger$ onto $N^\dagger$. Now we can perform another Lyapunov-Schmidt reduction on the space of odd functions:

\begin{lemma}[\cite{Kielhofer2012}]
Solving the infinite problem
\begin{align*}
	F^{q,\dagger}(0,p) = 0
\end{align*}
is equivalent to solving
\begin{align*}
	\Phi^\dagger(v,p) = 0,
\end{align*}
where $\Phi^\dagger\colon U^\dagger\times V \to Z^\dagger_0$ and $(v,p)\in U^\dagger\times V\subset N^\dagger\times\mathcal P$. Here, $\Phi^\dagger(0,p_0) = 0$ and $\Phi^\dagger$ is defined by
\begin{align}\label{eq:Phi_red_def}
	\Phi^\dagger(v,p) := Q^\dagger F^{q,\dagger}(v + \psi^\dagger(v,p),p),
\end{align}
where $\psi^\dagger\colon N^\dagger \times V\to X^\dagger_0$ is a unique function satisfying $\psi^\dagger(0,p_0) = 0$. It is implicitly defined to be the unique solution of the equation
\begin{align}\label{eq:psi_red_def}
	(I-Q^\dagger)F^{q,\dagger}(v+\psi^\dagger(v,p),p)=0
\end{align}
in a neighborhood of $(0,p_0)$.
\end{lemma}
This Lemma follows from~\cite{Kielhofer2012}. We can use it to show the next lemma:

\begin{lemma}\label{lem:hatPhi_noangle}
$\hat\Phi$ does not change the angle of a vector but only multiplies its length by a (possibly negative) factor. To be precise, for all $(a,b)$ in a small neighborhood $U$ of $(0,0)$ and $s\in(-\delta,\delta)$,
\begin{align}\label{eq:hatPhi_multiply}
    \hat \Phi\left(\begin{pmatrix}a\\b\end{pmatrix},s\right) = \hat h((a,b)^\top,s) \begin{pmatrix}a\\b\end{pmatrix},
\end{align}
where $\hat h\colon U\times(-\delta,\delta)\to \R$ is a rotationally invariant function, i.e.
\begin{align*}
    \hat h\left( A_\phi \begin{pmatrix}a\\b\end{pmatrix}, s\right) = \hat h\left(\begin{pmatrix}a\\b\end{pmatrix}, s\right)
\end{align*}
for all $\phi$.
\end{lemma}

\begin{proof}
First note, that due to Proposition~\ref{prop:hatPhi_symm}, it suffices to show~\eqref{eq:hatPhi_multiply} for $b=0$. Therefore, it is left to show that $\hat\Phi_2((a,0)^\top,s)=0$.
Since both $\psi$ and $\psi^\dagger$ are uniquely defined $\psi^\dagger$ must be the restriction of $\psi$ to the space $O$. In particular,
\begin{align*}
	\psi(v,p) = \psi^\dagger(v,p) \in O
\end{align*}
whenever $v\in O$. This shows that $\psi(au_\ell, p)$ is an odd function. Now, we evaluate $\hat\Phi_2((a,0)^\top,s)$:
\begin{align*}
	\hat \Phi_2((a,0)^\top,s) &= \langle z_2^*, \Phi(a u_\ell, p(s))\rangle\\
	&=\langle z_2^*, QF^q(au_\ell + \psi(a u_\ell, p(s)), p(s))\rangle\\
	&=0,
\end{align*}
because $au_\ell + \psi(au_\ell, p(s))$ is odd and $F^q$ maps odd functions to odd functions, see Lemma~\ref{lem:F_odd}.
Therefore, when $b=0$ in~\eqref{eq:hatPhi_def} and~$a$ and~$s$ are in a small neighborhood of the bifurcation point, we find that $\hat \Phi_2 = 0$. Since $\hat\Phi((0,0)^\top,s)=(0,0)^\top$, we can choose~$\hat h$ such that the claim of the lemma holds. Finally, by Proposition~\ref{prop:hatPhi_symm} it follows that~$\hat h$ has to satisfy the rotational invariance condition.
\end{proof}

Consequently, when looking for zeros of $\hat \Phi((a,b)^\top,s)$, we can restrict ourselves to $b=0$. Given $a,s$ such that $\hat \Phi((a,0)^\top,s)=0$, all other zeros can then be obtained by applying $A_\phi$ to $(a,0)^\top$.
Therefore, we might as well study the problem of finding zeros of
\begin{align*}
    h(a,s) := \hat \Phi((a,0)^\top,s) = \hat \Phi^\dagger(a,s) := \langle z_1^*, \Phi^\dagger(au_\ell, p(s))\rangle.
\end{align*}
In the next section, we Taylor-expand~$h$ to see which zeros it has in a neighborhood of the origin.

\subsection{Taylor Expansion around the Bifurcation Point}\label{sec:h}

In order to determine the type of the bifurcation it is necessary to compute the derivatives of~$h$. Since~$F$ is smooth, as proven in Appendix~\ref{sec:derivative}, the function~$\psi^\dagger$,~$\Phi^\dagger$ and~$\hat\Phi^\dagger$, which originate from the implicit function theorem or are concatenations of smooth functions, are smooth as well. In order to derive expressions for the derivative of $h$ we first need to compute derivatives of $F^{q,\dagger}$. These derivatives are given in the next lemma:

\begin{lemma}\label{lem:F_ders}
The derivative of $F^{q,\dagger}$, evaluated on the basis functions satisfies
\begin{align}\label{eq:F_der1u}
	F^{q,\dagger}_v(0,p)[u_k] = c_1(q,k,p)u_k.
\end{align}
An evaluation of second derivatives of $F^{q,\dagger}$ on the basis elements~$u_k$ yields
\begin{align}
	\label{eq:F_der2uu}
	F^{q,\dagger}_{vv}(0,p)[u_k,u_k] &= c_2(q,k,p) u_{2k}
\end{align}
Further, mixed second derivatives are given by
\begin{align}
	\label{eq:F_der2mixed_uu}
	F^{q,\dagger}_{vv}(0,p)[u_m, u_k] &= c_3(q,m,k,p) u_{m-k} + c_4(q,m,k,p) u_{m+k},
\end{align}
for all $m,k\in \N$ with $m\neq k$. Here, we use the convention $u_{-n} = -u_n$ for $n\in\N$. Furthermore, we find
\begin{align}
	\label{eq:F_der3uuu}
	F^{q,\dagger}_{vvv}(0,p)[u_k,u_k,u_k] &= 3c_5(q,k,p) u_k + c_6(q,k,p) u_{3k},
\end{align}
for $k\in \N$. Here, $c_1,\dots,c_6$ are coefficients that depend on $q,k,m,p$ and the Fourier coefficients $\hat W_r(k)$ of the coupling function. The full expressions for these coefficients can be found in Appendix~\ref{sec:abbreviations}.
\end{lemma}

\begin{proof}
    This lemma can be proven by inserting the basis functions into the representation of the derivative of~$F$, derived in Appendix~\ref{sec:derivative}.
\end{proof}

Now, we can use Lemma~\ref{lem:F_ders} to calculate derivatives of $\psi^\dagger$ and $\Phi^\dagger$. This follows the lines of~\cite[Section I.6]{Kielhofer2012}.

\begin{lemma}
The derivatives of $\psi^\dagger$ satisfy
\begin{align}
    \label{eq:psi_der1res}
	\psi^\dagger_v(0,p_0)u_\ell &= 0,\\
    \label{eq:psi_der2uu}
	\psi^\dagger_{vv}(0,p_0)[u_\ell,u_\ell] &= - \frac{c_2(q,\ell,p_0)}{c_1(q,2\ell,p_0)} u_{2\ell}.
\end{align}
\end{lemma}

\begin{proof}

Taking the derivative of~\eqref{eq:psi_red_def} with respect to $v$ yields
\begin{align}\label{eq:psi_der1}
	(I-Q^\dagger)F^{q,\dagger}_v(v + \psi^\dagger(v,p),p)[v_1 + \psi^\dagger_v(v,p)v_1] = 0
\end{align}
for all $v_1\in N^\dagger$. Now, we insert $v=0$ and $p = p_0$ into~\eqref{eq:psi_der1}. Noting that $F^{q,\dagger}_v(0,p_0)v_1 = 0$ and $Q^\dagger F^{q,\dagger}_v(0,p_0)=0$, we are left with $F^{q,\dagger}_v(0,p) \psi^\dagger_v(0,p_0)v_1 = 0$. Since $\psi^\dagger$ maps $N^\dagger$ into $X^\dagger_0$ and $F^{q,\dagger}_v(0,p)$ regarded as a map from $X^\dagger_0$ to $R^\dagger$ is bijective, we obtain~\eqref{eq:psi_der1res}.

Differentiating~\eqref{eq:psi_der1} once more with respect to~$v$ gives
\begin{align*}
	&(I-Q^\dagger)F^{q,\dagger}_{vv}(v+\psi^\dagger(v,p),p)[v_1 + \psi^\dagger_v(v,p)v_1, v_2 + \psi^\dagger_v(v,p)v_2]\\
	 &\quad+ (I-Q^\dagger)F^{q,\dagger}_v(v + \psi^\dagger(v,p),p)\psi^\dagger_{vv}(v,p)[v_1,v_2] = 0.
\end{align*}
for all $v_1,v_2\in N^\dagger$. Again, by inserting $v=0$ and $p = p_0$ into the previous equation we obtain
\begin{align}\label{eq:psi_der2res}
	(I-Q^\dagger)F^{q,\dagger}_{vv}(0,p_0)[v_1,v_2] + F^{q,\dagger}_v(0,p_0)\psi^\dagger_{vv}(0,p_0)[v_1,v_2] = 0
\end{align}
for all $v_1,v_2\in N^\dagger$. Now, we compute $\psi_{vv}^\dagger(0,p_0)[u_\ell,u_\ell]$ by choosing $v_1=v_2=u_\ell$ in~\eqref{eq:psi_der2res} and using~\eqref{eq:F_der2uu}. We obtain
\begin{align*}
	F^{q,\dagger}_v(0,p_0)\psi^\dagger_{vv}(0,p_0)[u_\ell,u_\ell] = -c_2(q,\ell,p_0)u_{2\ell}.
\end{align*}
Therefore, by noting that $\psi^\dagger_{vv}(0,p_0)[u_\ell,u_\ell]\in X^\dagger_0$, considering $F^{q,\dagger}_v(0,p_0)\colon X_0^\dagger\to R^\dagger$ as an invertible map and using~\eqref{eq:F_der1u}, we are left with~\eqref{eq:psi_der2uu}.
\end{proof}

Now we can use these derivatives to calculate the derivatives of~$\hat \Phi^\dagger$:
\begin{lemma}\label{lem:hatPhi_der}
$\hat\Phi^\dagger$ satisfies
\begin{align*}
    \hat\Phi^\dagger(0,0)&=0,   &\hat\Phi^\dagger_a(0,0)&=0,\\
    \hat\Phi^\dagger_{aa}(0,0)&=0, &\hat\Phi^\dagger_{aaa}(0,0)&=6\gamma_1,
\end{align*}
where
\begin{align}\label{eq:gamma1def}
	\gamma_1 := \frac{1}{2} \left( c_5(q,\ell,p_0) - \frac{c_2(q,\ell,p_0)c_3(q,2\ell,\ell,p_0)}{c_1(q,2\ell,p_0)} \right).
\end{align}
\end{lemma}

\begin{proof}
By differentiating~\eqref{eq:Phi_red_def} with respect to $v$ we get
\begin{align}
	\label{eq:Phi_der1gen}
	\Phi^\dagger_v(v,p)v_1 &= Q^\dagger F^{q,\dagger}_v(v+\psi^\dagger(v,p),p)[v_1 + \psi^\dagger_v(v,p)v_1],\\
	\nonumber
	\begin{split}
	\Phi^\dagger_{vv}(v,p)[v_1,v_2] &= Q^\dagger F^{q,\dagger}_{vv}(v+\psi^\dagger(v,p)[v_1 + \psi^\dagger_v(v,p)v_1,v_2 + \psi^\dagger_v(v,p)v_2]\\
	& \quad + Q^\dagger F^{q,\dagger}_v(v+\psi^\dagger(v,p),p)\psi^\dagger_{vv}(v,p)[v_1,v_2],
	\end{split}\\
	\nonumber
	\begin{split}
	\Phi^\dagger_{vvv}(v,p)[v_1,v_2,v_3] &= Q^\dagger F^{q,\dagger}_{vvv}(v + \psi^\dagger(v,p),p)[v_1 + \psi_v^\dagger(v,p)v_1, v_2 + \psi_v^\dagger(v,p)v_2, v_3 + \psi^\dagger_v(v,p)v_3]\\
	&\quad + Q^\dagger F^{q,\dagger}_{vv}(v + \psi^\dagger(v,p),p)[v_1 + \psi^\dagger_v(v,p)v_1, \psi^\dagger_{vv}(v,p)[v_2,v_3]]\\
	&\quad + Q^\dagger F^{q,\dagger}_{vv}(v + \psi^\dagger(v,p),p)[v_2 + \psi^\dagger_v(v,p)v_2, \psi^\dagger_{vv}(v,p)[v_1,v_3]]\\
	&\quad + Q^\dagger F^{q,\dagger}_{vv}(v + \psi^\dagger(v,p),p)[v_3 + \psi^\dagger_v(v,p)v_3, \psi^\dagger_{vv}(v,p)[v_1,v_2]]\\
	&\quad + Q^\dagger F^{q,\dagger}_v(v + \psi^\dagger(v,p),p)\psi^\dagger_{vvv}(v,p)[v_1,v_2,v_3]
	\end{split}
\end{align}
for all $v_1,v_2,v_3\in N^\dagger$. Evaluating these derivatives at $v=0$ and $p = p_0$ and using~\eqref{eq:psi_der1res} yields
\begin{subequations}
\begin{align}
	\label{eq:Phi_der1}
	\Phi^\dagger_v(0,p_0)v_1 &= 0,\\
	\label{eq:Phi_der2}
	\Phi^\dagger_{vv}(0,p_0)[v_1,v_2] &= Q^\dagger F^{q,\dagger}_{vv}(0, p_0)[v_1,v_2],\\
	\label{eq:Phi_der3}
	\begin{split}
	\Phi^\dagger_{vvv}(0,p_0)[v_1,v_2,v_3] &= Q^\dagger F^{q,\dagger}_{vvv}(0,p_0)[v_1,v_2,v_3]\\
	&\quad  + Q^\dagger F^{q,\dagger}_{vv}(0,p_0)[v_1,\psi^\dagger_{vv}(0,p_0)[v_2,v_3]]\\
	&\quad + Q^\dagger F^{q,\dagger}_{vv}(0,p_0)[v_2,\psi^\dagger_{vv}(0,p_0)[v_1,v_3]]\\
	&\quad + Q^\dagger F^{q,\dagger}_{vv}(0,p_0)[v_3,\psi^\dagger_{vv}(0,p_0)[v_1,v_2]].
	\end{split}
\end{align}
\end{subequations}
By the definition of~$\Phi^\dagger$ we get
\begin{align*}
	\hat\Phi^\dagger(0,0) &= 0
\end{align*}
By~\eqref{eq:Phi_der1} we get
\begin{align*}
	\hat\Phi^\dagger_a(0,0) &=  \langle z_1^*, \Phi_v(0,p_0)u_\ell\rangle = 0.
\end{align*}
Using~\eqref{eq:Phi_der2} and~\eqref{eq:F_der2uu}, we obtain
\begin{align*}
	\hat \Phi^\dagger_{aa}(0,0) &= \langle z_1^*, \Phi^\dagger_{vv}(0,p_0)[u_\ell,u_\ell]\rangle = \langle z_1^*, Q^\dagger F^{q,\dagger}_{vv}(0, p_0)[u_\ell,u_\ell]\rangle = \langle z_1^*, 0\rangle =  0.
\end{align*}
Using~\eqref{eq:Phi_der3},~\eqref{eq:F_der3uuu},~\eqref{eq:psi_der2uu} and~\eqref{eq:F_der2mixed_uu} yields
\begin{align*}
	\Phi^\dagger_{vvv}(0,p_0)[u_\ell,u_\ell,u_\ell] &= Q^\dagger F^{q,\dagger}_{vvv}(0,p_0)[u_\ell,u_\ell,u_\ell]\\
	&\quad + 3Q^\dagger F^{q,\dagger}_{vv}(0,p_0)[u_\ell, \psi^\dagger_{vv}(0,p_0)[u_\ell,u_\ell]]\\
	&=Q^\dagger F^{q,\dagger}_{vvv}(0,p_0)[u_\ell,u_\ell,u_\ell]\\
	&\quad - \frac{3c_2(q,\ell,p_0)}{c_1(q,2\ell,p_0)}Q^\dagger F^{q,\dagger}_{vv}(0,p_0)[u_\ell, u_{2\ell}]\\
	&=3 c_5(q,\ell,p_0)u_\ell -\frac{3c_2(q,\ell,p_0)c_3(q,2\ell,\ell,p_0)}{c_1(q,2\ell,p_0)} u_\ell\\
\end{align*}
Therefore, $\hat\Phi^\dagger_{aaa}(0,0) = 6\gamma_1$.

\end{proof}

Now, we compute derivatives involving~$s$.

\begin{lemma}\label{lem:hatPhi_mixedder}
$\hat\Phi^\dagger$ satisfies
\begin{align*}
    \hat \Phi^\dagger_s(0,0)=0, \qquad\hat \Phi^\dagger_{as}(0,0) = \gamma_2,
\end{align*}
where 
\begin{align}\label{eq:gamma2def}
	\gamma_2 := \frac{\d}{\d s}c_1(q,\ell,p(s))\Big|_{s=0}.
\end{align}
\end{lemma}

\begin{proof}

Since $F^{q,\dagger}(0,p) = F(\Psi^q,p) = 0$ for all $p\in\mathcal P$, we have $\hat\Phi^\dagger(0,s) = 0$ for all $s\in (-\delta,\delta)$. In particular,
\begin{align*}
	\hat \Phi_s(0,0) = 0.
\end{align*}
To compute the mixed derivative, we first differentiate~\eqref{eq:psi_red_def} with respect to~$p$ to obtain
\begin{align}\label{eq:psidef_pder}
	(I-Q^\dagger)F^{q,\dagger}_v(v+\psi^\dagger(v,p),p)\psi_p(v,p) + (I-Q^\dagger)F^{q,\dagger}_p(v+\psi^\dagger(v,p),p) = 0.
\end{align}
Now, we insert $v=0$, $p=p_0$ to get
\begin{align*}
	(I-Q^\dagger)F^{q,\dagger}_v(0,p_0)\psi^\dagger_p(0,p_0) + (I-Q^\dagger)F^{q,\dagger}_p(0,p_0) = 0.
\end{align*}
Again, because $F^{q,\dagger}(0,p) = 0$ for all~$p$ its derivative with respect to~$p$, i.e., the second part of the previous equation, is~$0$. Moreover for $v\in X_0^\dagger$, $(I-Q^\dagger)F^{q,\dagger}_v(0,p_0)v=0$ is equivalent to $v=0$ and therefore
\begin{align}\label{eq:psi_dp}
	\psi^\dagger_p(0,p_0) = 0.
\end{align}
Now, the mixed first derivatives can be computed by differentiating~\eqref{eq:Phi_der1gen} with respect to $p$ as follows:
\begin{align*}
	\Phi^\dagger_{vp}(v,p)[v_1,p_1] &= Q^\dagger F^{q,\dagger}_{vv}(v+\psi^\dagger(v,p),p)[v_1 + \psi_v^\dagger(v,p)v_1]\psi^\dagger_p(v,p)p_1\\
	&\quad + Q^\dagger F^{q,\dagger}_{vp}(v+\psi^\dagger(v,p),p)[v_1 + \psi_v^\dagger(v,p)v_1, p_1]\\
	&\quad + Q^\dagger F^{q,\dagger}_{v}(v+\psi^\dagger(v,p),p)\psi^\dagger_{vp}(v,p)[v_1,p_1]
\end{align*}
for all $v_1\in N^\dagger, p_1\in \R^3$. Evaluating that at $v=0$, $p=p_0$ yields
\begin{align*}
\Phi^\dagger_{vp}(0,p_0)[u_\ell,p_1] &= Q^\dagger F^{q,\dagger}_{vp}(0,p_0)[u_\ell,p_1]\\
&=D_p c_1(q,k,p)\Big|_{p=p_0}p_1u_\ell.
\end{align*}
Consequently,
\begin{align*}
	\hat\Phi_{as}(0,0) &= D_pc_1(q,k,p)\Big|_{p=p_0}p'(0)\\
	&=\frac{\d}{\d s}c_1(q,\ell,p(s))\Big|_{s=0}.
\end{align*}
\end{proof}

Now, we can put these lemmas together and formulate the concluding theorem of this section:
\begin{theorem}\label{thm:TaylorExpansion}
The Taylor-expansion of $\hat \Phi^\dagger(a,s) = h(a,s)$ is
\begin{align}
\nonumber
	\hat\Phi^\dagger(a,s) &= \left[
	\gamma_1 a^3 +\mathcal O(a^4)\right]
	+s\left[\gamma_2 a +\mathcal O(a^2)\right]+\mathcal O(s^2)\\
	\label{eq:hatPhi_cartesian}
	&=a(\gamma_1 a^2 + \gamma_2 s) + \mathcal O(a^4+\abs{s}a^2+s^2)
\end{align}
Here, $\gamma_1$ and $\gamma_2$ are defined as in~\eqref{eq:gamma1def} and~\eqref{eq:gamma2def}, respectively.
\end{theorem}
\begin{proof}
    Since~$F$ is smooth, this follows from Lemmas~\ref{lem:hatPhi_der} and~\ref{lem:hatPhi_mixedder}.
\end{proof}

Here,~\eqref{eq:hatPhi_cartesian} is the Taylor expansion of a pitchfork bifurcation. In fact, by using the implicit function theorem, one can show that except for the trivial solution branch $a=0$ there is another curve of equilibria in the neighborhood of the trivial solution. This non-trivial solution branch can be parameterized by a twice continuously differentiable curve $\tau\mapsto (s(\tau), a(\tau))$ for $\tau \in (-\epsilon,\epsilon)$ and $0<\epsilon$ small enough. In this case, $s(0) = a(0) = 0$ and the parameterization can be chosen such that $\dot a(0)=1$. Then, $\dot s(0)=0$ and $\ddot s(0) = -2\gamma_1/\gamma_2$. Here, a dot means differentiation with respect to $\tau$. Given these derivatives, the non-trivial solution curve exists for $s\le 0$ when $\gamma_2/\gamma_1>0$ and for $s\ge 0$ when $\gamma_2/\gamma_1<0$. Moreover, in a neighborhood of the bifurcation point, we can approximate
\begin{align}\label{eq:a_app}
    a\approx a^\mathrm{app}(s) := \sqrt{\frac{-\gamma_2 s}{\gamma_1}}.
\end{align}

Given a $a_0,s_0$ with $h(a_0,s_0) = 0$ we infer all $\tilde a_0, \tilde b_0, \tilde s_0$ with $\tilde a_0 ^2 + \tilde b_0^2 = a_0^2$ and $s_0 = \tilde s_0$ satisfy $\hat \Phi((\tilde a_0, \tilde b_0)^\top, \tilde s_0) = 0$, due to the symmetry~\eqref{eq:hatPhi_symmetry} of $\hat \Phi$.

\vspace{1cm}

\subsection{Higher-Order Equilibria Approximations}\label{sec:higher_expansion}

In the last section, we clarified existence of solutions to $\Phi(v,p(s))=0$. In this section we explain how to use these solutions to derive formulas that can be used to approximate the zeros of~$F^q$ in~$H_0^1$. Given $v\in N$ that solves $\Phi(v,p(s)) = 0$ we know that $F(\Psi^q + v + \psi(v,p(s)),p(s))=0$. For a given~$p(s)$, a zero of~$F$ is therefore given by $Z(v,s) := \Psi^q + v + \psi(v,p(s))$.
A naive~$0$-th order approximation would be given by
\begin{align*}
	Z(v,s) = \Psi^q + \mathcal O(\norm{(v,s)}).
\end{align*}
However, since after neglecting the higher-order terms this approximation coincides with the trivial zero of~$F$, i.e., the $q$-twisted state, this approximation is not useful.

An approximation of first order can be derived by expanding $Z(v,s)$ in terms of~$v$ and~$s$ up to first derivatives. This yields
\begin{align*}
	Z(v,s) &= \Psi^q + v + \psi_v(0,p_0)v + s\psi_p(0,p_0)p'(0) + \mathcal O(\norm{(v,s)}^2)\\
	&=\Psi^q + v + \mathcal O(\norm{(v,s)}^2),
\end{align*}
where we have used that $\psi_v(0,p_0) = 0$ and $\psi_p(0,p_0) = 0$. Neglecting the higher-order terms, we denote
\begin{align}\label{eq:1st_order_expansion}
	Z^1(v,s) = \Psi^q + v
\end{align}
for the first order approximation.

To get a more precise approximation, we assume that $p\colon (-\delta,\delta)\to \mathcal P$ is a smooth curve. Then, we expand up to second order:
\begin{align*}
	\Psi^q + v + \psi(v,p(s)) = \Psi^q + v + \frac 12 (v, p(s)-p(0)) H \begin{pmatrix}
	v\\ p(s)-p(0)
	\end{pmatrix} + O(\norm{(v,s)}^3),
\end{align*}
with
\begin{align*}
 H = \begin{pmatrix}
	\psi_{vv}(0, p_0)[v,v] & \psi_{vp}(0,p_0)[v, p'(0)]\\
	\psi_{vp}(0,p_0)[v, p'(0)] & \psi_{pp}(0,p_0)[p'(0), p'(0)] + \psi_p(0,p_0)p''(0)
	\end{pmatrix}
\end{align*}
First, we take care of the lower right entry of~$H$. Note that $\psi_p(0,p_0) = 0$, as shown in~\eqref{eq:psi_dp}. Next, differentiating~\eqref{eq:psidef_pder} with respect to~$p$ and evaluating at $p = p_0$ and $v = 0$ yields
\begin{subequations}
\begin{align}
\label{eq:psidef_dpp1}
0 &= (I-Q) F_{\Psi\Psi}(\Psi^q,p_0)[\psi_p(0,p_0)p'(0),\psi_p(0,p_0)p'(0)]\\
\label{eq:psidef_dpp2}
&\quad + (I-Q) F_{\Psi p}(\Psi^q,p_0)[\psi_p(0,p_0)p'(0), p'(0)]\\
\label{eq:psidef_dpp3}
&\quad + (I-Q) F_\Psi(\Psi^q,p_0)\psi_{pp}(0,p_0)[p'(0), p'(0)]\\
\label{eq:psidef_dpp4}
&\quad + (I-Q) F_{\Psi p}(\Psi^q,p_0)[\psi_p(0,p_0)p'(0), p'(0)]\\
\label{eq:psidef_dpp5}
&\quad + (I-Q) F_{pp}(\Psi^q, p_0)[p'(0), p'(0)].
\end{align}
\end{subequations}
Again, due to $\psi_p(0,p_0) = 0$, we observe that the terms~\eqref{eq:psidef_dpp1},~\eqref{eq:psidef_dpp2} and~\eqref{eq:psidef_dpp4} equal $0$. Moreover, $F(\Psi^q,p) = 0$ for all $p\in \mathcal P$. Therefore, $F_{pp}(\Psi^q,p_0) = 0$, and thus~\eqref{eq:psidef_dpp5} is $0$ as well. As a consequence
\begin{align*}
	0 = (I-Q) F_\Psi(\Psi^q,p_0)\psi_{pp}(0,p_0)[p'(0), p'(0)].
\end{align*}
Since~$\psi$ maps into~$X_0$ we conclude that $\psi_{pp}(0,p_0)[p'(0), p'(0)] = 0$ and thus $H_{22}=0$.

Second, we look at the off-diagonal entries $H_{21}=H_{12}$. To obtain an expression for $\psi_{vp}(0,p_0)$ we differentiate~\eqref{eq:psi_der1} with respect to~$p$, insert $v = 0$, $p = p_0$ and thereby obtain
\begin{subequations}
\begin{align}
\label{eq:psidef_dpv1}
	0 &= (I-Q) F_{\Psi\Psi}(\Psi^q,p_0)[v_1, \psi_p(0,p_0)p'(0)]\\
\label{eq:psidef_dpv2}
	&\quad + (I-Q) F_{\Psi p}(\Psi^q,p_0)[v_1, p'(0)]\\
\label{eq:psidef_dpv3}
	&\quad + (I-Q) F_\Psi(\Psi^q,p_0)\psi_{vp}(0,p_0)[v_1,p'(0)],
\end{align}
\end{subequations}
for all $v_1\in N$. Since $\psi_p(0,p_0) = 0$, as shown in~\eqref{eq:psi_dp},~\eqref{eq:psidef_dpv1} equals $0$. Moreover, for all $p\in \mathcal P$, $F_\Psi(\Psi^q,p)v\in Z_0$ for all $v\in N$. Therefore, $F_{\Psi p}(\Psi^q,p)[v,p'(0)]\in Z_0$, too, and consequently~\eqref{eq:psidef_dpv2} is~$0$. Again, we conclude that $\psi_{vp}(0,p_0) = 0$.
Therefore,
\begin{align*}
	H = \begin{pmatrix}
	\psi_{vv}(0,p_0)[v,v]& 0\\ 0&0
	\end{pmatrix}
\end{align*}
and thus
\begin{align*}
	\Psi^q + v + \psi(v, p(s)) = \Psi^q + v + \frac 12 \psi_{vv}(0,p_0)[v,v] + O(\norm{(v,s)}^3)
\end{align*}
for all $v\in N$ and $s\in (-\delta,\delta)$ in a neighborhood of $(0,0)$. We denote
\begin{align}\label{eq:2nd_order_expansion}
	Z^2(v,s) = \Psi^q + v + \frac 12 \psi_{vv}(0,p_0)[v,v]
\end{align}
for the second order expansion.

However, until now we have assumed that~$v$ solves $\Phi(v,p(s))=0$. Since this solution~$v$ depends on~$s$, we denote it by~$v(s)$. Unfortunately, for given~$s$ these~$v$ are not known exactly but they also have to be approximated by a function that we call $v^\mathrm{app}$. We derive~$v^\mathrm{app}$ by first computing~$a^\mathrm{app}(s)$ according to~\eqref{eq:a_app}. Then, $\hat \Phi^\dagger(a^\textrm{app}(s),s)\approx 0$ and consequently $\hat \Phi(A_\phi (a^\textrm{app}(s),0)^\top,s)\approx 0$ for all~$\phi$, but for simplicity we keep $\phi = 0$. Now, we define $v^\mathrm{app}(s):=a^\mathrm{app}(s)u_\ell$ and because $\hat \Phi$ is the coordinate version of $\Phi$, it follows that $\Phi(v^\mathrm{app}(s),s)\approx 0$. Given this function~$v^\mathrm{app}(s)$, we use $Z^i(v^\mathrm{app}(s),s)$ for $i\in\{1,2\}$ to approximate the real equilibrium $Z(v(s),s)$. Consequently, the total approximation error is given by
\begin{subequations}
\begin{align}
    \label{eq:app_step1}
    \abs{Z^i(v^\mathrm{app}(s),s)-Z(v(s),s)} &\le \abs{Z^i(v^\mathrm{app}(s),s)-Z^i(v,s)}\\
    \label{eq:app_step2}
    &\quad+ \abs{Z^i(v(s),s)-Z(v(s),s)}.
\end{align}
\end{subequations}
In the remaining part of this subsection, we determine the magnitude of the approximation error of both parts~\eqref{eq:app_step1} and ~\eqref{eq:app_step2} in dependence of the parameter~$s$.

To obtain an estimate for the first part~\eqref{eq:app_step1}, we reconsider the curve $(a(\tau), s(\tau))$ that describes the nontrivial equilibria. Because $\dot a(0)=1$, we can reparameterize the curve such that locally $a(\tau) = \tau$. Then, we still have $s(0) = 0$, $\dot s(0) = 0$ and $\ddot s(0)=-2\gamma_1/\gamma_2$. Due to $\dot a(\tau) = 1$ for all~$\tau$ in a small neighborhood of the origin and the symmetry of $\hat\Phi$ discussed in Lemma \ref{lem:hatPhi_noangle}, we can infer that $s(-\tau) = s(\tau)$. Consequently, $s(\tau)$ has vanishing third derivative at $\tau=0$. Since the curve is smooth that results in $s(\tau) = -\gamma_1/\gamma_2 \tau^2 + \mathcal O(\tau^4)$. Using this representation, one can show that
\begin{align*}
    \abs{a^\textrm{app}(s(\tau))-a(\tau)} = \abs{\sqrt{\frac{-\gamma_2 s(\tau)}{\gamma_1}}-\tau} = \mathcal O(\tau^3) = \mathcal O(s^\frac 32).
\end{align*}
Since $Z^i(v,s)$ is polynomial in~$v$, that then results in
\begin{align*}
    Z^i(v^\mathrm{app}(s),s)-Z^i(v(s),s) = \mathcal O(s^\frac 32).
\end{align*}

To estimate the second part of the error~\eqref{eq:app_step2}, it is important to note that ~$v(s)$ is dependent on~$s$. In fact, due to the pitchfork bifurcation, its dependence can be expressed as $v(s) = \mathcal O(s^\frac 12)$. Combining that with $Z^i(v,s)-Z(v,s) = \mathcal O(\norm{(v,s)}^{i+1})$, as shown above, we find
\begin{align*}
    Z^i(v(s),s)-Z(v(s),s) = \mathcal O(\norm{(v(s),s)}^{i+1}) = \mathcal O(s^\frac{i+1}{2}).
\end{align*}
Putting these two errors together, we conclude that the total approximation error is given by
\begin{align*}
    Z^i(v^\mathrm{app}(s),s)-Z^i(v(s),s) = \mathcal O(s^{\min(\frac 32, \frac{i+1}2)}).
\end{align*}

In particular, deriving a third order approximation $Z^3(v,s)$ or even higher-order approximations is useless unless one can also improve the approximation in the first step~\eqref{eq:app_step1}. This, however, would require a more detailed Taylor-expansion of $\hat \Phi^\dagger$ than the one given in Theorem~\ref{thm:TaylorExpansion} and thus more derivatives of~$F^q$.

\subsection{Linear Stability}\label{sec:stability}

Up to now, we have only determined the existence of equilibria of the PDE
\begin{align*}
	\frac{\partial}{\partial t} \Psi(t,x) = F(\Psi, p)(x).
\end{align*}
We have seen that apart from the trivial solution, there exists a solution curve of nontrivial solutions. In this section, we formally investigate the linear stability of $q$-twisted states and bifurcating branches. A rigorous proof of nonlinear stability is beyond the scope of this article.

Without loss of generality, we assume $\gamma_2>0$. If this is not the case, reverse the parameterization of $p(s)$ by considering $p(-s)$ instead. Moreover, since the stability depends on the spectrum of the linearization, we denote
\begin{align*}
	\kappa(s) := \sup_{\stackrel{k\in\N}{k\neq \ell}} c_1(q,k,p(s))
\end{align*}
and assume $\kappa(0) <0$ since otherwise neither the twisted state nor the bifurcating equilibria can be stable. Note that $\kappa(s)$ is continuous in~$s$ and thus $\kappa(s)<0$ for all~$s$ in a neighborhood of $0$. Consequently, we only have to investigate how the critical zero eigenvalues at the bifurcation change, when perturbing $(\Psi^q,p_0)$ to nearby equilibria.

\paragraph{Stability of the twisted state.}
Since $\gamma_2 = \frac{\d}{\d s}c_1(q,\ell,p(s))\Big|_{s=0}$ is assumed to be positive, $c_1(q,\ell,p(s))<0$ for all $s<0$. Consequently, $\sup_{k\in\N} c_1(q,k,p(s))<0$ for all $s<0$, which means that the spectrum of $F_v^q(0,p(s))$ is in the left half of the complex plane. Therefore, $\Psi^q$ is linearly stable. If, on the other hand $s>0$, $F_v^q(0,p(s))$ has positive eigenvalues, from which we can conclude linear instability.

\paragraph{Stability of the bifurcating branches in~$O$.}
First, we study the stability of the bifurcating equilibria only in the space of odd functions~$O$. Considering the bifurcation problem in this reduced space, there is only one critical eigenvalue with multiplicity one that passes through~$0$ and a one-dimensional curve of bifurcating equilibria. As explained at the end of Section~\ref{sec:h}, this curve corresponds to $\tau\mapsto (s(\tau), a(\tau))$ for $\tau \in (-\epsilon,\epsilon)$ with $s(0) = a(0) = 0$ and $\dot a(0)=1$. Further, we denote $v(\tau) = a(\tau) u_\ell + \psi^\dagger(a(\tau)u_\ell, p(s(\tau)))$ for the equilibrium of $F^{q,\dagger}$ such that $F^{q,\dagger}(v(\tau), p(s(\tau)))=0$. The principle of exchange of stability~\cite[Section I.7]{Kielhofer2012} can now be applied to study the linear stability of these bifurcating equilibria. First, the critical zero eigenvalue gets perturbed to an eigenvalue~$\nu(\tau)$ of $F^{q,\dagger}_v(v(\tau), p(s(\tau)))$, see~\cite[Proposition I.7.2]{Kielhofer2012}. To be precise,
\begin{align*}
	F^{q,\dagger}_v(v(\tau), p(s(\tau)))(u_\ell + \omega(\tau)) =  \nu(\tau) (u_\ell + \omega(\tau)),
\end{align*}
where $\omega(\tau)\in O$ is a continuously differentiable curve, $\nu(0)=0$ and $s\in (-\delta,\delta)$. Moreover, $\nu(\tau)$ is continuously differentiable and represents the perturbation of the zero eigenvalue. Its derivative at $\tau = 0$ can be computed using the formula
\begin{align*}
	\gamma_2\dot s(0) = -\dot{ \nu} (0),
\end{align*} 
see also formula $(I.7.41)$ in~\cite{Kielhofer2012}. However, due to $\dot s(0) = 0$ we obtain $\dot{ \nu}(0)=0$. The second derivative satisfies
\begin{align*}
	2\gamma_2\ddot s(0) = -\ddot{\nu}(0),
\end{align*}
see formula $(I.7.45)$ in~\cite{Kielhofer2012}. Using $\ddot s(0) = -2\gamma_1/\gamma_2$ we find $\ddot{ \nu}(0) = 4\gamma_1$. Since $\kappa(s)<0$, the stability of the bifurcating branch in a neighborhood of the bifurcation point is then determined by the sign of the perturbed eigenvalue~$\nu(\tau)$. To be precise, if $\gamma_1<0$ then $\ddot s(0) > 0$ and thus bifurcating solutions exist whenever $s>0$ is close to~$0$. Since $\gamma_2<0$, the $q$-twisted state has a positive eigenvalue and is thus linearly unstable in that parameter region. The leading eigenvalue of the bifurcating solution~$v(\tau)$, however is given by $\nu(\tau)<0$. Therefore, the bifurcating solutions are stable. In this case, the bifurcation is supercritical. If $\gamma_1>0$, we have $\ddot s(0)<0$. Consequently, the bifurcating solutions exist for $s<0$. Here, the $q$-twisted state is linearly stable and the leading eigenvalue of the bifurcating branch is $\nu(\tau)>0$. Thus, these bifurcating equilibria are unstable. Such a bifurcation is called subcritical.

\paragraph{Stability of the bifurcating branches in~$H^1_0$.} 
Now, we consider the bifurcation problem in~$H^1_0$. First note, that since $O\subset H_0^1$, the equilibrium $v(\tau)$ is still an equilibrium of~$F^q$ when considered in~$H_0^1$. Furthermore, by applying the symmetry condition~\eqref{eq:F_symmetry}, one can retrieve every other equilibria in a neighborhood of the bifurcation point. Specifically, for all $\phi\in \R$, the functions $B_\phi(v(\tau))$ are also equilibria. This symmetry results in a two-dimensional surface of equilibria, that is parameterized by~$\phi$ and~$\tau$. Corresponding to this surface of equilibria, there are two critical zero spectral values, that we need to track when perturbing the trivial equilibria $(0,p_0)$ to bifurcating equilibria $(v(\tau), p(s(\tau)))$ that lies on the surface. Obviously, since $O\subset H_0^1$, $(v(\tau), p(s(\tau)))$ inherits the eigenvalue $\nu(\tau)$. Because $(v(\tau), p(s(\tau)))$ lies on a surface of equilibria, the other spectral value is given by~$0$. Even though this zero spectral value prevents us from directly concluding linear stability, our numerical simulations in the next section show that bifurcating equilibria are stable in~$H_0^1$ when they are stable in~$O$.

\section{Applications}\label{sec:applications}

In this section, we take a few specific choices of the curve~$p(s)$ and evaluate the bifurcation in more detail. We compute the ratio $\gamma_2/\gamma_1$ which determines if bifurcating solutions exist for $s\ge 0$ or $s\le 0$. Moreover, we approximate these bifurcating solutions using the expansions in Section~\ref{sec:higher_expansion} and study their existence numerically as a cross-validation. In the first part of the section, we only look at graph coupling. The second subsection additionally includes one higher-order interaction and shows how that can influence the stability of twisted states. Finally, in the last part of this section, we consider all higher-order interactions and explain how they can be used to change the type of the bifurcation from subcritial to supercritical or vice versa.

\subsection{The Kuramoto Model on Nonlocal Graphs}\label{sec:application_graph}

\subsubsection{The Attractive Kuramoto Model (Subcritical Bifurcation)}

In this section we apply the bifurcation theory to the Kuramoto model on limits of $k$-nearest-neighbor graphs. Specifically, we consider no higher-order interactions, i.e., $\lambda = \mu = 0$ in~\eqref{eq:pd_problem}. Instead we only consider the coupling range~$r$ in the continuum limit as a parameter. When~$r>0$ is very small, the eigenvalues around a $q$-twisted state are all negative~\cite{Wiley2006}. Then, upon increasing~$r$, the eigenvalue corresponding to~$k=1$ is the first one that passes through~$0$. We denote this threshold by $r_0^\mathrm{a}(q)$ with a superscript~$\mathrm{a}$ to indicate that we are working with the attractive Kuramoto model. It is called attractive, since two oscillators that are close attract each other. In our notation that means $c_1(q,k,(r,0,0)) < 0$ for all $k\in \N$ and $r\in (0,r_0^\mathrm{a}(q))$ and $c_1(q,1,(r_0^\mathrm{a}(q),0,0))=0$, see Figure~\ref{fig:classical_eigvals} and~\cite{Wiley2006}.

\begin{figure}[h]
\centering
\begin{overpic}[grid=off, width = \textwidth]{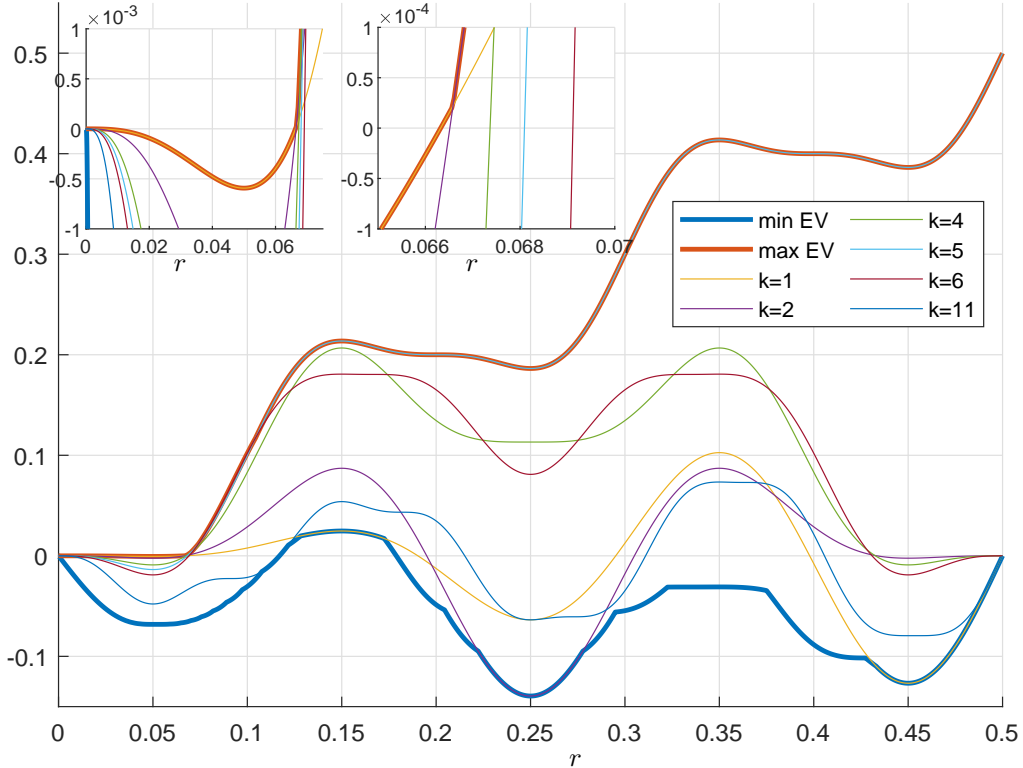}
\put(55,3){$r$}
\put(22.5,44){$r$}
\put(46.5,44){$r$}
\end{overpic}
\caption{Eigenvalues $c_1(5,k,(r,0,0))$ of a $5$-twisted state. Until $r_0^\mathrm{a}(5)\approx 0.06632$ all eigenvalues are negative. At $r_0^\mathrm{a}(5)$, the eigenvalue $c_1(5,k,(r,0,0))$ with $k=1$ passes through $0$. Shortly thereafter, the eigenvalue corresponding to $k=2$ passes through $0$ as well. For $0.1170 \lessapprox r \lessapprox 0.1789$ all eigenvalues are positive.}
\label{fig:classical_eigvals}
\end{figure}

To analyze this bifurcation we choose a curve $p\colon (-\delta,\delta)\to\mathcal P$ with $p(s) = (r_0^\mathrm{a}(q) + s, 0, 0)$. As explained in Section~\ref{sec:h}, finding equilibria of the Kuramoto model on a graph around the $q$-twisted state in a neighborhood of the bifurcation at $r_0^\mathrm{a}(q)$ is equivalent to finding solutions to the equation $\hat \Phi^\dagger(a, s) = 0$ in a neighborhood of the origin. According to the results in the same section, for given $s_0\in(-\delta,\delta)$, an approximate solution is given by $a^\mathrm{app}$ and it exists whenever the quantity under the root in~\eqref{eq:a_app} is positive. As seen in Figure~\ref{fig:classical_ratio}(a) and shown in Appendix~\ref{sec:gamma_ratio}, $\gamma_2/\gamma_1>0$. Moreover, since the first eigenvalue passes through~$0$ from below, we have $\gamma_2>0$, which then implies $\gamma_1>0$.
\begin{figure}[h]
\centering
\begin{overpic}[grid=off, width = .9\textwidth]{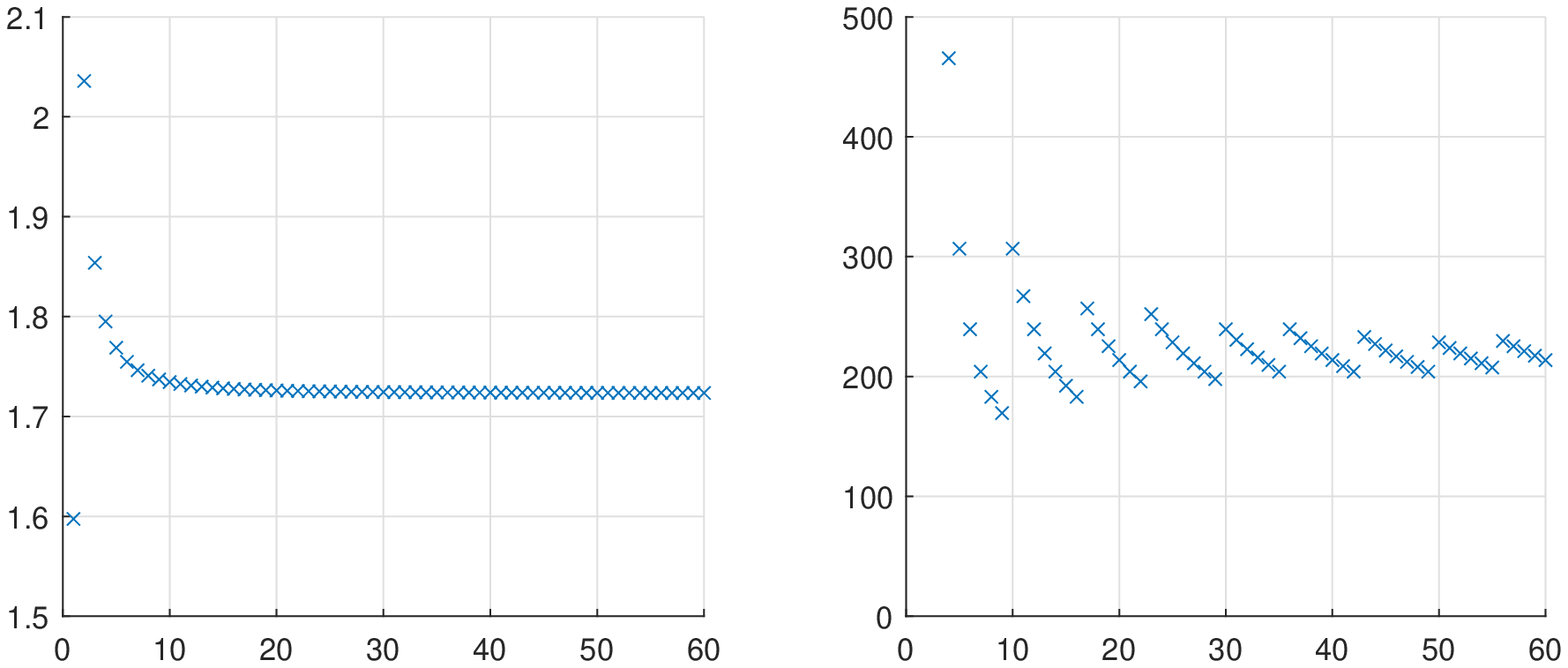}
\put(7,1){\textbf{(a)}}
\put(52,1){\textbf{(b)}}
\put(32,1){$q$}
\put(76,1){$q$}
\put(7,32){$\frac{\gamma_2}{\gamma_1q}$}
\put(50,32){$\frac{\gamma_2}{\gamma_1q}$}
\end{overpic}
\caption{Bifurcation ratio $\frac{\gamma_2}{\gamma_1q}$. \textbf{(a)} shows the ratio for the attractive Kuramoto model at $r = r_0^\mathrm{a}(q)$ when the first ($k=1$, independent of $q$) eigenvalue passes through~$0$. The value of $\frac{\gamma_2}{\gamma_1q}$ seems to converge to $\approx 1.723$. \textbf{(b)} shows the same ratio for the repulsive Kuramoto model at $r = r_0^\mathrm{r}(q)$, which is the bifurcation when the $q$-twisted first becomes stable upon increasing~$r$. Note that in the repulsive Kuramoto model the critical eigenvalue is not the same for all~$q$. In fact, $\operatorname{argmin}_k c_1(q,k,(r_0^\mathrm{r}(q),0,0))$ depends on~$q$ and jumps in \textbf{(b)} correspond to irregular changes of $\operatorname{argmin}$. For $q=1$, there is no bifurcation in the repulsive Kuramoto model and the data points that are approximately at $(2,6248)$ and $(3, 1045)$ are omitted in this plot.}
\label{fig:classical_ratio}
\end{figure}
Therefore, for $r\in(r_0^\mathrm{a}(q)-\delta,r_0^\mathrm{a}(q))$ there exist further equilibria of~\eqref{eq:problem} and~\eqref{eq:pd_problem} when $\lambda=\mu=0$ around the $q$-twisted states. However, according to the principles explained in Section~\ref{sec:stability}, the $q$-twisted state is stable in that regime and the bifurcating solutions are unstable.

To confirm the existence of the bifurcating solutions, we consider the sequence of finite particle systems~\eqref{eq:classical_finite} whose continuum limit is given by~\eqref{eq:continuum_graph}, or equivalently~\eqref{eq:problem} with $\lambda=\mu=0$. In these systems, the coefficients~$a_k$ are defined by $a_k = 1$ if $\min(\abs{k}, M-\abs{k})\le \lfloor M r\rfloor$ and $a_k=0$ otherwise. The corresponding system of phase differences, defined by $\theta_k := \phi_k - \phi_1$, is given by
\begin{align}\label{eq:classical_finite_disc_pd}
	\dot\theta_k = \frac 1M \sum_{j=1}^M a_{k-j} \sin(\theta_j-\theta_k)-\frac 1M \sum_{j=1}^M a_{1-j} \sin(\theta_j).
\end{align}
However, since the definition of~$a_k$ involves rounding,~$r$ cannot be regarded as a continuous bifurcation parameter. Therefore, we consider the system 
\begin{align}\label{eq:classical_finite_c_pd}
	\dot\theta_k = \frac 1M \sum_{j=1}^M b_{k-j} \sin(\theta_j-\theta_k)-\frac 1M \sum_{j=1}^M b_{1-j} \sin(\theta_j),
\end{align}
in which the coefficients $b_k$ are defined as follows: Let $k_0 = \lfloor rM\rfloor$. Then,
\begin{align*}
	b_k:= \begin{cases}
	1&\quad \text{ if } \min(\abs{k}, M-\abs{k}) \le k_0\\
	rM-k_0&\quad \text{ if }\min(\abs{k}, M-\abs{k}) = k_0+1\\
	0	&\quad \text{ otherwise }
	\end{cases}
\end{align*}
Here,~$r$ can be considered as a continuous bifurcation parameter.

However, when simulating the finite particle system~\eqref{eq:classical_finite_c_pd}, it turns out that the  bifurcation does not occur at $r_0^\mathrm{a}(q)$ but at another value $r_0^{\mathrm{a},M}(q)$ which is slightly different from $r_0^\mathrm{a}(q)$. In fact, numerical simulations show $r_0^\mathrm{a}(q) = r_0^{\mathrm{a},M}(q) + \mathcal O(1/M)$. For example, for $q=5$ we get $r_0^\mathrm{a}(5) \approx 0.06632$ whereas $r_0^{\mathrm{a}, 1000}(5) \approx 0.06582$. Consequently, when looking for bifurcating solutions of $q$-twisted in the finite particle system~\eqref{eq:classical_finite_c_pd} one should search in a neighborhood of $r_0^{\mathrm{a},M}(q)$. In particular, we fix $s = s_0$ and look for bifurcating solutions for $r = r_0^{\mathrm{a},M}(q)+s_0$. In order to get an approximation for a solution of $\Phi(v,p(s))=0$, we calculate $\gamma_1$ and $\gamma_2$ according to~\eqref{eq:gamma1def} and~\eqref{eq:gamma2def} based on the value $r=r_0^\mathrm{a}(q)+s_0$. Then, we calculate $a^\mathrm{app}(s_0)$ according to~\eqref{eq:a_app} and proceed by along the steps explained in Section~\ref{sec:higher_expansion} to get $v^\mathrm{app}$. We next use a discrete analog of the first order approximation $Z^1(v^\mathrm{app},s_0)$ as the initial condition of a zero finding algorithm (e.g. a Newton-iteration), that we then apply to the right-hand side of~\eqref{eq:classical_finite_c_pd}. For one specific parameter choice, the solution~$\hat Z$ of this zero-finding algorithm is depicted in Figure~\ref{fig:classical_nontrivialsolution}.

\begin{figure}[h]
\centering
\begin{overpic}[grid=off, width = 1\textwidth]{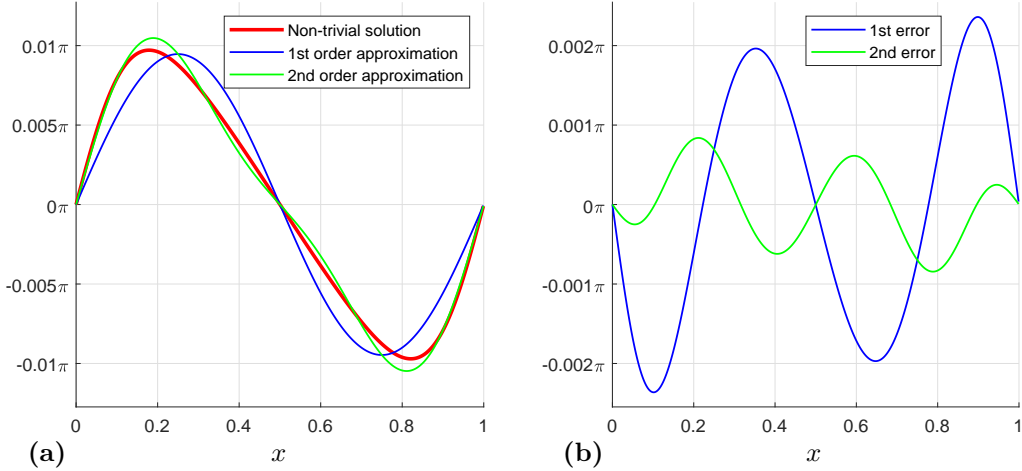}
\put(9,0){\textbf{(a)}}
\put(53,0){\textbf{(b)}}
\put(29,0){$x$}
\put(73,0){$x$}
\end{overpic}
\caption{Bifurcating solution around a $q$-twisted state in the system~\eqref{eq:classical_finite_c_pd}. \textbf{(a)} shows $\hat Z-\Psi^q$ (red line), its first order approximation $Z^1(v^\mathrm{app},s_0)-\Psi^q$ (blue) and its second order approximation $Z^2(v^\mathrm{app},s_0)-\Psi^q$ (green). Here $v^\mathrm{app} = a^\mathrm{app} u_1$ such that $\Phi(v^\mathrm{app},p(s_0))\approx 0$ up to higher-order terms. It this was exact, as explained in Section~\ref{sec:higher_expansion}, we would have $F(Z(v,s_0),p(s_0))=0$. We denote~$\hat Z$ for an equilibrium of the system~\eqref{eq:classical_finite_c_pd} for $r = r_0^{\mathrm{a},M}(q) +s_0$. Moreover, \textbf{(b)}  depicts the first order error $\hat Z-Z^1(v^\mathrm{app},s_0)$ (blue line) and the second order error $\hat Z-Z^2(v^\mathrm{app},s_0)$ (green line). Parameter values: $M=1000$, $s_0 = -10^{-4}$, $q = 5$. That results in $\gamma_1 \approx 9.494 \cdot 10^{-3}, \gamma_2 \approx 8.400 \cdot 10^{-2}, a^\mathrm{app} = 2.974 \cdot 10^{-2}$.}
\label{fig:classical_nontrivialsolution}
\end{figure}
While the simulations depicted in Figure~\ref{fig:classical_nontrivialsolution} are based on the system~\eqref{eq:classical_finite_c_pd}, in which~$r$ is a continuous bifurcation parameter, similar results hold for the system~\eqref{eq:classical_finite_disc_pd}. Here, however, we could not choose~$M$ arbitrarily. Instead, we were particularly successful finding bifurcating solutions when $r_0^{\mathrm{a},M}(q)$ is close to an integer multiple of $1/M$.

\subsubsection{The Repulsive Kuramoto Model (Supercritical Bifurcation)}

Now we consider a variant of~\eqref{eq:pd_problem} with $\lambda = \mu = 0$, in which we reverse the sign of the right-hand side. In particular, we look at
\begin{align}\label{eq:repulsive_continuum_pd}
	\frac{\partial}{\partial t}\Psi(t,x) &= -\int_\S W_r(x-y)\sin(\Psi(t,y)-\Psi(t,x))\ \d y + \int_\S W_r(y)\sin(\Psi(y))\ \d y.
\end{align}
Its finite-dimensional analog of~\eqref{eq:classical_finite_c_pd} is then given by
\begin{align}\label{eq:repulsive_finite_c_pd}
	\dot\theta_k = \frac{-1}M \sum_{j=1}^M b_{k-j} \sin(\theta_j-\theta_k)+\frac{1}M \sum_{j=1}^M b_{1-j} \sin(\theta_j).
\end{align}
Here, two oscillators that are close to each other, repel each other. Therefore, we call this model the \emph{repulsive} Kuramoto model, see also~\cite{Girnyk2012}.

Because the systems~\eqref{eq:repulsive_continuum_pd} and~\eqref{eq:pd_problem} with $\lambda=\mu=0$ are only different by a factor of $-1$ on the right-hand side, they share the same equilibria. Yet, they are not identical, since the stability of these equilibria depends on the eigenvalues of the linearization of the right-hand side and they are nonidentical. In fact, the spectrum of the linearization of the right-hand side of~\eqref{eq:repulsive_continuum_pd} can be obtained from multiplying the spectrum of~\eqref{eq:pd_problem} by~$-1$. Similarly, the eigenvalues of the system~\eqref{eq:repulsive_finite_c_pd} can be obtained by multiplying the eigenvalues from~\eqref{eq:classical_finite_c_pd} by~$-1$. Therefore, the eigenvalues of the linearization of the right hand around a $q$-twisted state are $-c_1(q,k,(r,0,0))$. Consequently, the $q$-twisted state is linearly stable if $c_1(q,k,(r,0,0))>0$ for all $k\in\N$.

As seen in Figure~\ref{fig:classical_eigvals}, for $q=5$ there is an interval $r\in (0.117, 0.1789)$ in which all these conditions are satisfied. Moreover, in a neighborhood of the lower boundary of this interval, $\max_k -c_1(5,k,(r,0,0)) = -\min_k c_1(5,k,(r,0,0))$ is attained for $k=11$. It was shown in~\cite{Girnyk2012} that for all $q\in\N$ there is an interval in which a $q$-twisted state is linearly stable in the repulsive Kuramoto model. In particular, the authors showed that stability holds if $1.1787\lessapprox 2qr \lessapprox 1.7829$, which agrees with our observation in Figure~\ref{fig:classical_eigvals}. For $q=1$ there is no bifurcation in the repulsive Kuramoto model.

To analyze the bifurcation, we denote $r_0^\mathrm{r}(q)$ for the smallest value of~$r$ until which there is a positive eigenvalue of the $q$-twisted state in the repulsive Kuramoto model. For example $r_0^\mathrm{r}(5) \approx 0.11787$. We then choose the parameter curve $p(s) = (r_0^\mathrm{r}(q) + s,0,0)$. A numerical evaluation of~\eqref{eq:gamma1def} and~\eqref{eq:gamma2def} for $q = 5$ and $k = 11$ shows $\gamma_1, \gamma_2 > 0$, see also Figure~\ref{fig:classical_ratio}(b). Therefore, bifurcating solutions exist when $s<0$, or equivalently $r<r_0^\mathrm{r}(5)$. Since the $5$-twisted state is unstable in that regime, these bifurcating solutions are linearly stable.

To validate that numerically, we first choose~$M$ large enough and then determine $r_0^{\mathrm{r},M}(5)$, which is the bifurcation point in the~$M$-particle system with $\lim_{M\to \infty} r_0^{\mathrm{r},M}(5)\to r_0^\mathrm{r}(5)$. Next, we choose $s < 0$ with small enough~$\abs{s}$. Finally, we simulate the system~\eqref{eq:classical_finite_c_pd} for random initial conditions that are close to the $5$-twisted state until the system reaches an equilibrium. We observe, see Figure~\ref{fig:classical_repulsive_odesol}, that all resulting equilibria lie in a neighborhood of the $5$-twisted state. The differences between these equilibria and the $5$-twisted state are sinusoidal functions with~$11$~periods on the domain~$\S$, which corresponds to the unstable direction spanned by~$u_{11}$ and~$w_{11}$ of the $5$-twisted state. Moreover, the amplitudes of these differences are all the same and they can be approximated by solving~\eqref{eq:a_app}.

\begin{figure}
\centering
\begin{overpic}[grid=off, width = 1\textwidth]{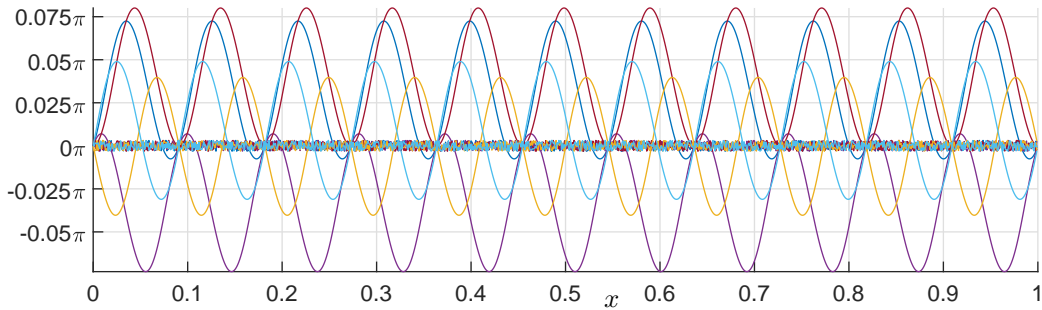}
\put(55,0){$x$}
\end{overpic}
\caption{Simulation of~\eqref{eq:classical_finite_c_pd} with initial conditions that are close to the $5$-twisted state until trajectories have reached an equilibrium. The wiggly lined that are centered around ~$0$ are random functions with amplitude~$10^{-2}$ in whose direction we have perturbed the of the $5$-twisted state to obtain the initial condition for the ODE solver. The resulting limiting equilibria of the ODE system are depicted as modulations of the $5$-twisted state. These limiting equilibria can be obtained by adding the sinusoidal functions to the $5$-twisted state. All sinusoidal functions have an amplitude of $\approx 0.12575  \approx 0.04 \pi$ which is close to the prediction $a^\mathrm{app} = 0.0394 \pi$ by~\eqref{eq:a_app}. Furthermore, these modulations can be obtained from each other by applying the operator $B_\phi$. Parameter values: $M = 1000, s = -10^{-5}, r_0^{\mathrm{r},M}(5) \approx 0.11654, r_0^\mathrm{r}(5) = 0.11704, r = r_0^{\mathrm{r},M}(q) + s$, $\gamma_1 \approx 1.38*10^{-3}, \gamma_2 \approx 2.12$.} 
\label{fig:classical_repulsive_odesol}
\end{figure}

\subsection{Stabilization via Higher-Order Interactions}\label{sec:stabilization}

In this section, we keep $r_0\in (0,\frac 12)$ and $\mu_0 = 0$ constant and vary $\lambda$. When considering the eigenvalues $c_1(q,k,(r_0,\lambda,0))$ we note that $\frac{\d}{\d \lambda}c_1(q,k,(r_0,\lambda,0))$ is constant with respect to $\lambda$ and does not depend on~$k$. We denote this quantity by $h(q,r_0)$. Note that it is given by
\begin{align*}
	h(q,r_0) = -\hat W_{r_0}(q) = -\frac{2}{\pi q}\sin(2\pi q r_0).
\end{align*}
Therefore, whenever this quantity is nonzero, one can use higher-order interactions to stabilize or destabilize $q$-twisted states on the continuum limit of $k$-nearest-neighbor graphs. To explain this, suppose for example, that $m:= \max_k c_1(q,k,(r_0,0,0)) > 0$. In this case the $q$-twisted state is unstable in the model~\eqref{eq:pd_problem} with $\lambda=\mu=0$. Due to the linearity of $c_1(q,k,(r_0,\lambda,0))$ with respect to~$\lambda$, we can then write
\begin{align*}
	c_1(q,k,(r_0,\lambda,0)) = c_1(q,k,(r_0,0,0)) + \lambda h(q,r_0).
\end{align*}
Then, the maximal eigenvalue of the linearization around a $q$-twisted state for parameters $r_0,\lambda$ is given by $m+\lambda h(q,r_0)$. Consequently, if a $q$-twisted state is unstable for $\lambda = 0$, i.e., $m>0$, one can stabilize it by choosing $\lambda < \frac{-m}{h(q,r_0)}$ if $h(q,r_0)>0$ and $\lambda > \frac{-m}{h(q,r_0)}$ if $h(q,r_0)<0$. In the nongeneric case $h(q,r_0)=0$, a (de)stabilization is not possible.

\begin{figure}[h]
\centering
\begin{overpic}[grid=off, width = 1\textwidth]{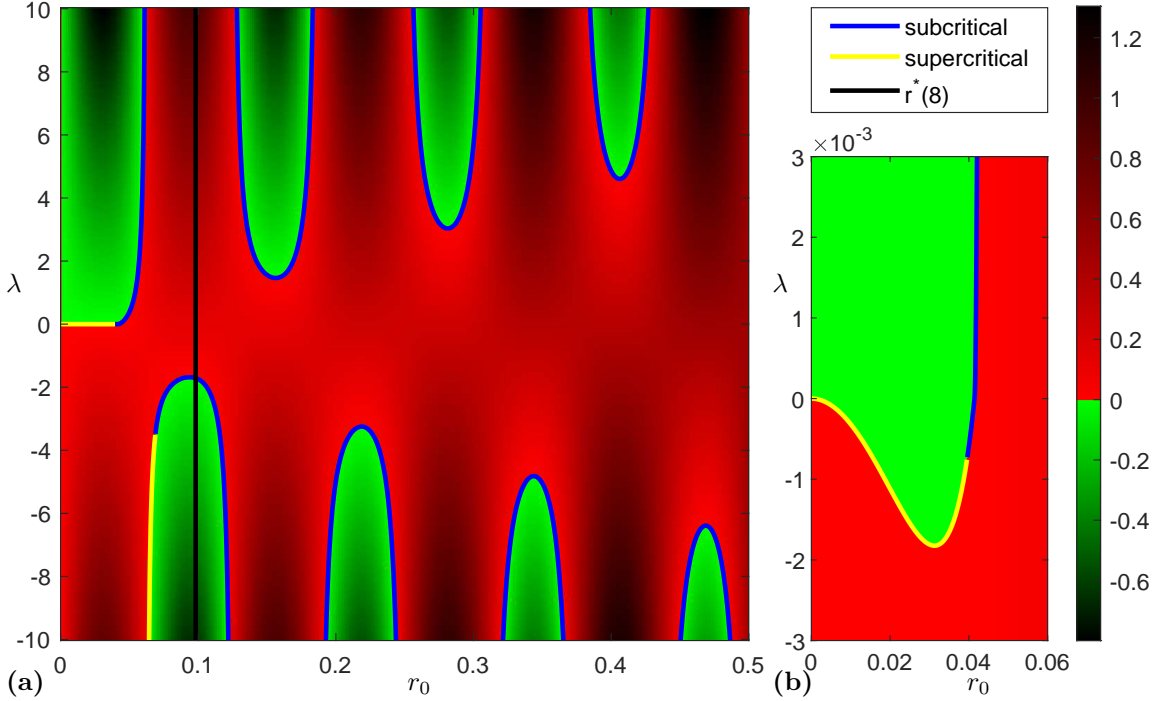}
\put(5,-1){\textbf{(a)}}
\put(68,-1){\textbf{(b)}}
\put(38,-1){$r_0$}
\put(84,-1){$r_0$}
\put(5,32){$\lambda$}
\put(68,32){$\lambda$}
\end{overpic}
\caption{Maximal eigenvalue of a $8$-twisted state for in the system~\eqref{eq:pd_problem} in dependence of~$r_0$ and~$\lambda$. Red colors represent an unstable $8$-twisted state, whereas the $8$-twisted state is stable if the color is green. The bigger the magnitude of the maximal eigenvalue, the darker the color. The black line depicts the value of $r^\star(8)$. For all $r_0>r^\star(8)$ the bifurcation boundary can be calculated from~\eqref{eq:higher_order_lambda0}. The blue curve indicates a subcritical bifurcation whereas supercritical bifurcations are yellow. \textbf{(b)} is a more detailed view of one region in \textbf{(a)}. Note also that the bifurcation at $\lambda=0$ is subcritical, as shown in Section~\ref{sec:application_graph}} 
\label{fig:higher_order_eigvals}
\end{figure}

It can be shown (see Theorem~\ref{thm:grapheigenvalues}) that for all~$q$ and large enough~$r$ the largest eigenvalue $m = \max_k c_1(q,k,(r,0,0))$ is attained for $k=q$. In particular, a sufficient condition that the largest eigenvalue~$m$ is attained for $k=q$ is
\begin{align}\label{eq:largest_eigval_qk_condition}
	\frac{2}{\pi q} \le 2r - \frac 1\pi \sin(2\pi r),
\end{align}
which is proven in Appendix~\ref{sec:sup}. Since the right-hand side of~\eqref{eq:largest_eigval_qk_condition} is monotonically increasing in~$r$, there is a threshold $\tilde r(q)$ such that~\eqref{eq:largest_eigval_qk_condition} holds for all $r>\tilde r(q)$. Moreover, due to the continuity of the right-hand side and the convergence of the left-hand side to~$0$ as $q\to \infty$, this threshold converges to~$0$ as $q\to \infty$. However, since~\eqref{eq:largest_eigval_qk_condition} is only a sufficient condition, the largest eigenvalue might already be attained by $k=q$ for $r<\tilde r(q)$. We denote $r^\star(q)$ for the smallest value of $r\in (0,\frac 12)$ such that for all $r\ge r^\star(q)$, the largest eigenvalue is attained for $q=k$. We then have the inequality $0\le r^\star(q)\le \tilde r(q)\le \frac 12$. 
Consequently, when $r_0\in(r^\star(q),\frac 12]$ and $\hat W_{r_0}(q)\neq 0$ there is a bifurcation when $\lambda = \lambda_0$, with
\begin{align}\label{eq:higher_order_lambda0}
	\lambda_0 = -\frac{m}{h(q,r_0)} = \frac{\hat W_{r_0}(2q) + \hat W_{r_0}(0) - 2\hat W_{r_0}(q)}{4\hat W_{r_0}(q)}.
\end{align}

\begin{remark}
	Note that we need $\lim_{k\to\infty} c_1(q,k,(r_0,\lambda_0,0)) = -(\frac 12 + \lambda_0)\hat W_{r_0}(q) \neq 0$ in order to ensure that an isolated eigenvalue is passing through~$0$ at the bifurcation point. When $r_0>r^\star(q)$ this is equivalent to $\lambda_0\neq-\frac 12$, since we assumed $\hat W_{r_0}(q)\neq 0$. However, since $\hat W_{r_0}(2q)+\hat W_{r_0}(0)>0$, it guaranteed by~\eqref{eq:higher_order_lambda0} that $\lambda_0\neq -\frac 12$.
\end{remark}

By analyzing the ratio of $\gamma_1/\gamma_2$ we can determine the type of the bifurcation, see Figure~\ref{fig:higher_order_eigvals}.

\subsection{Changing the Type of the Bifurcation}\label{sec:bif_change}

Here, we fix~$r$ and consider~$4\lambda + 2\mu$ as the bifurcation parameter. Then, we vary~$\lambda$ and see how this variation affects the type of the bifurcation. Assuming~\eqref{eq:largest_eigval_qk_condition}, the largest eigenvalue is attained for $k=q$ and the bifurcation takes place at
\begin{align*}
4\lambda + 2\mu = \frac{\hat W_r(0) + \hat W_r(2q)-2\hat W_r(q)}{\hat W_r(q)} =: H(q,r).
\end{align*}
It is easy to see that~\eqref{eq:largest_eigval_qk_condition} is satisfied if $2 \le 2\pi qr-\sin(2\pi q r)$, which is equivalent to $qr~\ge~\upsilon_0~\approx~0.4065$, where $\upsilon_0$ solves $2 = 2\pi \upsilon_0 - \sin(2\pi \upsilon_0)$.
We consider the curve
\begin{align*}
	p^\mathfrak{t}(s) = (r^\mathfrak{t}(s), \lambda^\mathfrak{t}(s), \mu^\mathfrak{t}(s))^\top = (r_0, 4s-2\mathfrak{t}+H(q,r_0)/4, 2s+4\mathfrak{t})^\top,
\end{align*}
which is parameterized by~$s$ and $\mathfrak{t}\in \R$ is an additional parameter. Note that $4\lambda^\mathfrak{t}(0) + 2\mu^\mathfrak{t}(0) = H(q,r)$ for all $\mathfrak{t}\in\R$. Therefore, there is a bifurcation at $s=0$ for all $\mathfrak{t}\in \R$.
We calculate
\begin{align*}
	\gamma_2^\mathfrak{t} = \frac{\d}{\d s}c_1(q,q,p)\Big|_{s=0} = -5\hat W_r(q),
\end{align*}
which is independent of~$\mathfrak{t}$. Moreover,
\begin{align*}
	c_5(q,q,p^\mathfrak{t}(0)) &= c_5(q,q,p^0(0))-\frac 12 \mathfrak{t} \hat W_r(q)\\
	c_2(q,q,p^\mathfrak{t}(0)) &= c_2(q,q,p^0(0))-\frac 12 \mathfrak{t} (-\hat W_r(0)+\hat W_r(2q))\\
	c_3(q,2q,q,p^\mathfrak{t}(0)) &= c_3(q,2q,q,p^0(0))\\
	c_1(q,2q,p^\mathfrak{t}(0)) &= c_1(q,2q,p^0(0))
\end{align*}
which leads to
\begin{align*}
	\gamma_1^\mathfrak{t} &= \frac{1}{2} \left(c_5(q,q,p^\mathfrak{t}(0))-\frac{c_2(q,q,p^\mathfrak{t}(0))c_3(q,2q,q,p^\mathfrak{t}(0))}{c_1(q,2q,p^\mathfrak{t}(0))}\right)\\
	&=\gamma_1^0 - \frac{1}{4}\mathfrak{t}  \hat W_r(q) + \frac{c_3(q,2q,q,p^0(0))}{4c_1(q,2q,p^0(0))}\mathfrak{t}(-\hat W_r(0)+\hat W_r(2q))\\
	&= \gamma_1^0 + \mathfrak{t} X(q,r)
\end{align*}
with
\begin{align*}
X(q,r) = -\frac 14 \hat W_r(q) + \frac{c_3(q,2q,q,p^0(0))}{4c_1(q,2q,p^0(0))}(-\hat W_r(0)+\hat W_r(2q)).
\end{align*}

It can be shown (see Appendix~\ref{sec:abbreviations}) that
\begin{align}\label{eq:X_j_relation}
	X(q,r) = \frac 1q \iota(qr) 
\end{align}
for a function $\iota\colon \R_{\ge 0}\to \R$, see Figure~\ref{fig:jfun}. Moreover, based on the explicit expression of~$\iota$ that is given in the appendix, we conclude $\iota(\upsilon) = \upsilon + \mathcal O(1)$ as $\upsilon\to \infty$.

\begin{figure}[h]
\centering
\begin{overpic}[grid=off, width = 0.5\textwidth]{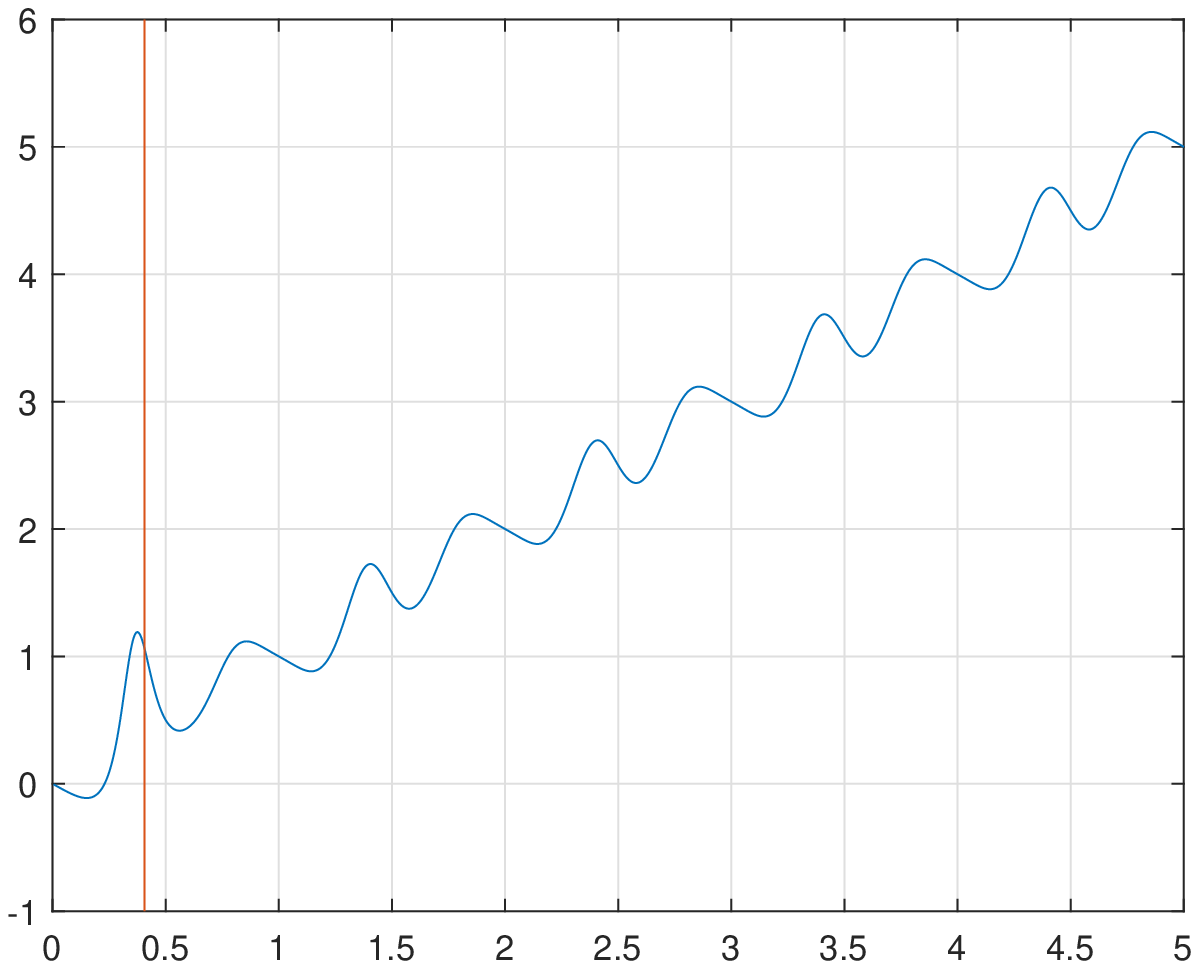}
\put(55,2){$\upsilon$}
\put(18,2){$\upsilon_0$}
\put(0,45){$\iota(\upsilon)$}
\end{overpic}
\caption{The function $\iota(\upsilon)$.} 
\label{fig:jfun}
\end{figure}

To conclude, $\gamma_2^\mathfrak{t}$ is independent of~$\mathfrak{t}$ and whenever $\gamma_2^\mathfrak{t}\neq 0$ there is a bifurcation at $s=0$. Furthermore, if additionally $X(q,r)\neq 0$, $\gamma_1^\mathfrak{t}$ can take any value in $\R$ by suitably choosing $\mathfrak{t}$. Consequently, the ratio $\gamma_2^\mathfrak{t}/\gamma_1^\mathfrak{t}$ can also take arbitrary values in $\R\setminus \{ 0\}$. Since the sign of that ratio determines the type (sub- or supercritical) of the bifurcation, the parameter $\mathfrak{t}$ can be used to influence the type of the bifurcation. As can be rigorously shown, $\iota(\upsilon)>0$ for all $\upsilon \ge \upsilon_0$, see Figure~\ref{fig:jfun}. Therefore, a sufficient condition to have $X(q,r)>0$ is given by $qr \ge \upsilon_0$.

The possibility of changing the bifurcation type by adjusting the strengths of various higher-order interactions to the continuum limit of $k$-nearest-neighbor graphs extends previous results~\cite{Skardal2019, Kuehn2021a}. In particular,~\cite{Skardal2019} contains a global bifurcation analysis for a coupling composed of a pairwise and two higher-order interaction terms. The authors found that by suitably choosing the strengths of the higher-order interaction one can influence the type of the pitchfork bifurcation, in which a certain state changes its stability.
Moreover, in~\cite{Kuehn2021a} it is shown in the context of bifurcations from trivial branches in network dynamics, that modifications of the original network model can generically induce changes between sub- and super-critical bifurcations. The main idea in~\cite{Kuehn2021a} is to first formally study normal forms for suitable macroscopic observables and then monitor the effect of network model changes in the concrete bifurcation coefficients in the normal form. Hence, the results presented here concretely and rigorously prove that higher-order interactions can trigger the effect between sub- and supercritical bifurcations, even for nontrivial branches of twisted states.

\section{Other Higher-Order Nonlocal Couplings}\label{sec:honns}

In this section, we discuss other possible generalizations of the pairwise interactions~\eqref{eq:continuum_graph} to higher-order interactions. Instead of focusing on the system of phase differences, we consider the original system that describes the absolute position~$\Theta(t,x)$ of the oscillators.
The name ``nearest-neighbor" coupling for the graph interactions~\eqref{eq:classical_finite} originates from supposing that the individual oscillators $i = 1,\dots,M$ are equidistantly placed on the unit circle in ascending order. In a nearest-neighbor graph, each oscillator is then connected to its predecessor and its successor on the circle. More generally, in a $k$-nearest-neighbor graph, each oscillator is connected to all of its $k$-predecessors and $k$-successors on the circle. If one fixes $r = k/M$ and sends $M\to \infty$, a nonlocal coupling in the continuum limit emerges. One can imagine two oscillators $x,y\in \S$ in the continuum limit to be coupled if $W_r(x-y)=1$. The parameter~$r$ specifies how far the two oscillators can be spaced apart such that they are still considered neighbors.

Now, let us consider the higher-order coupling
\begin{align}\label{eq:way3_coupling}
	\frac{\partial}{\partial t} \Theta(t,x) = 	\int_\S \int_\S W_r(z+y-2x) \sin(\Theta(t,z)+\Theta(t,y) - 2\Theta(t,x)) \ \d y \d z,
\end{align}
which is characterized by the coupling kernel $W_r(z+y-2x)$. While this is a straight-forward generalization of the pairwise continuum limit of $k$-nearest-neighbor coupling~\eqref{eq:continuum_graph} to higher-order interactions, the terminology ``nearest-neighbor" has to be used more carefully. Strictly speaking, three oscillators $z,y,x\in \S$ do not need to neighbor each other for them to be coupled, which is the case when $W_r(z+y-2x)=1$. For example, when $z = y = \frac 12$ and $x = 0$, we have $z+y-2x=0\in\S$. Therefore, these three oscillators are coupled for every $r>0$, even though~$z$ and~$x$ are relatively far apart. Similar arguments also hold for the $4$-way interaction in~\eqref{eq:problem}. One can further generalize~\eqref{eq:way3_coupling} by replacing the coupling function and the interaction function to 
\begin{align}\label{eq:way3_gen_coupling}
	\frac{\partial}{\partial t} \Theta(t,x) = 	\int_\S \int_\S W_r(m_1z + m_2 y + m_3 x)~\sin(n_1\Theta(t,z) + n_2\Theta(t,y) + n_3\Theta(t,x)) \ \d y \d z,
\end{align}
or even to the $d+1$-way coupling
\begin{align}\label{eq:way3_gen_plus_coupling}
	\frac{\partial}{\partial t} \Theta(t,x) = \int_{\S^d} W_r\left( \sum_{i=1}^d m_i y_i + m_{d+1} x\right) ~\sin\left( \sum_{i=1}^d n_i \Theta(t,y_i) + n_{d+1}\Theta(t,x)\right) \ \d y
\end{align}
for coefficients $m_i,n_i\in \Z\setminus\{0\}$. Note, however, that we must have $n_i = m_i$, $i=1,\dots,d+1$ in order for $q$-twisted states to be invariant. For special cases---as those considered above---the system is symmetric: If $\sum_{i=1}^{d+1}n_i=0$ then we have a phase shift symmetry~\eqref{eq:symmetry_phase_differences} and if $\sum_{i=1}^{d+1}m_i=0$ then we have a rotational symmetry of the ring~\eqref{eq:symmetry_rotation}

If one wants to derive higher-order continuum limits that overcome the issue of the nearest-neighbor terminology in higher-order networks, one can consider other generalizations of~\eqref{eq:continuum_graph}. For example, three oscillators $z,y,x\in \S$ in the $3$-way coupling
\begin{align}\label{eq:way4_coupling}
	\frac{\partial}{\partial t} \Theta(t,x) = 	\int_\S \int_\S W_r(z-x)W_r(y-x) \sin(\Theta(t,z)+\Theta(t,y)-2\Theta(t,x))\ \d y \d z
\end{align}
are coupled if $z$ is close to $x$ and additionally $y$ is close to $x$. As a result, all three oscillators $z,y$ and $x$ need to be close enough to each other for them to be coupled. Another possible higher-order generalization of~\eqref{eq:continuum_graph} is given by
\begin{align}\label{eq:way5_coupling}
	\frac{\partial}{\partial t} \Theta(t,x) = 	\int_\S \int_\S W_r(z-x)W_r(y-x)W_r(z-y) \sin(\Theta(t,z)+\Theta(t,y)-2\Theta(t,x))\ \d y \d z.
\end{align}
This coupling additionally introduces a symmetry between $x,y$ and~$z$. In fact, if $x,y,z$ are coupled, then any permutation of them is also coupled. Note that the prefactors $1,1,-2$ of $\Theta(t,z),\Theta(t,y)$ and $\Theta(t,x)$ in~\eqref{eq:way4_coupling} and~\eqref{eq:way5_coupling} can also be generalized to arbitrary coefficients $n_1,n_2,n_3\in \Z\setminus \{0\}$. However, they must add up to zero, i.e., $n_1+n_2+n_3=0$. 
Such ``diffusive'' coupling terms guarantees the invariance of $q$-twisted states and can correspond, for example, to a normal form symmetry in a phase reductions~\cite{Ashwin2016a}.

Of course, we can also study the stability of $q$-twisted states in the models~\eqref{eq:way3_gen_plus_coupling},~\eqref{eq:way4_coupling} and~\eqref{eq:way5_coupling}. Calculating the eigenvalues of the linearization of the right-hand sides of those systems around a $q$-twisted state yields the following:

The eigenvalues of the linearization of the right-hand side of~\eqref{eq:way3_gen_plus_coupling} with $n_i = m_i,~i=1,\dots,d+1$ are given by $\lambda_0 = 0$ with multiplicity~$1$ and $\lambda_k = \frac{1}{2}m_{d+1}\hat W_r(q)$ if $k\neq 0$. This eigenvalue has multiplicity $\infty$.
In the system~\eqref{eq:way4_coupling}, the eigenvalues are $\lambda_0 = 0$, again with multiplicity~$1$ and
\begin{align*}
    \lambda_k =  \frac 14 \hat W_r(q) \Big[ \hat W_r(q+k) + \hat W_r(q-k) -2 \hat W_r(q)\Big]
\end{align*}
if $k\neq 0$, each with multiplicity $2$. Finally, for the system~\eqref{eq:way5_coupling}, the eigenvalues around a $q$-twisted state are given by $\lambda_0 = 0$ (multiplicity~$1$) and 
\begin{align*}
    \lambda_k = \frac 18 \sum_{\ell\in \Z} \Big [ &\hat W_r(-k+\ell-q)\hat W_r(\ell+q) + \hat W_r(-k+\ell+q)\hat W_r(\ell-q)\\
    &+  \hat W_r(\ell - q) \hat W_r(k + q+\ell) + \hat W_r(\ell + q) \hat W_r(k-q+\ell)\\
    &-4 \hat W_r(\ell-q)\hat W_r(\ell+q)  \Big]
\end{align*}
if $k \neq 0$. Again, the multiplicity of these eigenvalues is~$2$. For all these systems, the eigenfunctions are given by $\sin(2\pi k x)$ and $\cos(2\pi k x)$. 

In Section~\ref{sec:stabilization} we showed that adding higher-order interactions of the form~\eqref{eq:way3_coupling} to the pairwise coupling~\eqref{eq:continuum_graph} can stabilize $q$-twisted states, when the strength of the higher-order interactions is adjusted suitably.  A numerical analysis shows that the systems~\eqref{eq:way3_gen_plus_coupling},~\eqref{eq:way4_coupling} and~\eqref{eq:way5_coupling} can also stabilize $q$-twisted states when added to the pairwise coupling~\eqref{eq:continuum_graph}.
However, we chose two different instances of higher-order interactions of the form~\eqref{eq:way3_gen_plus_coupling} for a couple of reasons. First, the formulas for $3$-way and $4$-way coupling are simple, since they only include the evaluation of~$W_r$ once. In contrast, the $3$-way coupling~\eqref{eq:way5_coupling} involves three evaluations of~$W_r$ and a generalization to $4$-way coupling would involve even more evaluations. Second, it is easier to compute eigenvalues of the linearization of the right-hand side of~\eqref{eq:way3_coupling} around a twisted state than it is to compute them for~\eqref{eq:way4_coupling} and~\eqref{eq:way5_coupling}. Third and most importantly, these eigenvalues of~\eqref{eq:way3_gen_plus_coupling} are independent of~$k$. Therefore, when adding a system~\eqref{eq:way3_gen_plus_coupling} to the pairwise system~\eqref{eq:continuum_graph}, the maximal eigenvalue is still attained for the same~$k$ when varying the strength of the higher-order coupling.

\section{Conclusion}\label{sec:conclusion}

In this article, we considered the continuum limit of a Kuramoto model on $k$-nearest-neighbor networks and extended it to include higher-order interactions. We analyzed the stability of $q$-twisted states and performed a rigorous Lyapunov--Schmidt reduction to find bifurcating equilibria. We saw that the bifurcation at which the twisted states lose their stability is a pitchfork bifurcation. Moreover, we determined leading coefficients in the Taylor expansion to classify the bifurcation as sub- or supercritical.
This considers and extends previous works from two perspectives. Firstly, we added a bifurcation analysis to previous works~\cite{Wiley2006, Girnyk2012}, which have analyzed stability of twisted state in the attractive and repulsive Kuramoto model on $k$-nearest-neighbor graphs. In particular, the problem of finding and classifying bifurcating solutions was left open in~\cite{Wiley2006}. Secondly, the authors of~\cite{Skardal2019} considered a higher-order all-to-all Kuramoto model whose right-hand consists of a pairwise part and two higher-order interaction parts, thus resembling our model~\eqref{eq:pd_problem}. In this model they analyzed the stability of the splay state and the bifurcation at which it looses its stability. While there is always a pitchfork bifurcation, they found that when varying the strengths of the higher-order couplings, as we did in Section~\ref{sec:bif_change}, one can influence if the bifurcation is sub- or supercritical. In that regard, we have extended their results to the continuum limit of $k$-nearest-neighbor networks and thereby shown that the phenomenon that one can change the type of a bifurcation with suitable higher-order interactions generically occurs in a wider class of higher-order networks.

Our work has also raised further follow up questions. For example, we believe that our techniques are also applicable to analyze bifurcations of generalized twisted states in similar models. For example, one can analyze twisted states in the Kuramoto model on other graphs, whose limit can still be characterized by a function $W_r(x)$, which does not necessarily need to be of the form~\eqref{eq:W_def}. Since our formulas~\eqref{eq:hatPhi_cartesian} are also valid in that case, our work poses a framework within which one can study the effect of the Fourier coefficients of~$W_r(x)$ on the bifurcation. Another example to which our theory could be applied is generalized twisted states on a two-dimensional lattice whose stability has been analyzed in~\cite{Goebel2021}.
Moreover, apart from the pairwise coupling we have only considered $3$-way and $4$-way higher-order interactions, however one can certainly add $5$-way coupling, $6$-way coupling, etc. and investigate how these interactions influence the bifurcation. Then, one might expect to control even higher order derivatives of~\eqref{eq:hatPhi_cartesian}. Regarding this question from a different perspective, one can also ask if every polynomial that respects the symmetry of the system can be obtained as a Taylor expansion of the function $\Phi^\dagger(a,s) = h(a,s)$ when adding enough higher-order interactions.

\pagebreak

\appendix

\section{Abbreviations}\label{sec:abbreviations}
let $p = (r,\lambda,\mu)\in \mathcal P$. Using the convention $\hat W_r(-k) := \hat W_r(k)$, we then define

\begin{align*}
	c_1(q,k,p) &= \frac{1}{4}\Big(\hat W_r(q-k)+ \hat W_r(q+k) - 2\hat W_r(q) - (4\lambda+2\mu) \hat W_r(q)\Big)\\
	c_2(q,k,p) &= \frac 18 \Big( -\hat W_r(q-2k) + 2\hat W_r(q-k) - 2\hat W_r(q+k) + \hat W_r(q+2k)-2\lambda\hat W_r(q-k) + 2\lambda\hat W_r(q+k)\Big)\\
	c_3(q,m,k,p)&=\frac{1}{8}\Big(-\hat W_r(q-m) + \hat W_r(q-m+k)+\hat W_r(q-k)\\
	&\qquad- \hat W_r(q+k) - \hat W_r(q+m-k) + \hat W_r(q+m)\Big)\\
	c_4(q,m,k,p)&=\frac{1}{8}\Big(-\hat W_r(q-m-k) + \hat W_r(q-m) + \hat W_r(q-k)\\
	&\qquad - \hat W_r(q+k) - \hat W_r(q+m) +\hat W_r(q+m+k)\Big)\\
	c_5(q,k,p)&=\frac{1}{16}\Big( \hat W_r(q-2k) - 4 \hat W_r(q-k) + 6 \hat W_r(q) - 4 \hat W_r(q+k) + \hat W_r(q+2k)\\
	&\qquad + 4\lambda \hat W_r(q-k) + 32\lambda\hat W_r(q) + 4\lambda\hat W_r(q+k)\\
	&\qquad + 2\mu \hat W_r(q-k) + 14\mu\hat W_r(q) + 2\mu\hat W_r(q+k) \Big)\\
	c_6(q,k,p)&=\frac{1}{16}\Big( \hat W_r(q-3k) - 3\hat W_r(q-2k) + 3\hat W_r(q-k) - 2\hat W_r(q) + 3\hat W_r(q+k)\\
	&\qquad  - 3\hat W_r(q+2k) + \hat W_r(q+3k) -12 \lambda \hat W_r(q-k) - 16 \lambda \hat W_r(q)-12\lambda \hat W_r(q+k)\\
	&\qquad -2\mu \hat W_r(q)   \Big)
\end{align*}

\noindent
$c_2$ does not depend on $\mu$.\\
$c_3$ does not depend on $\lambda, \mu$. $c_3$ satisfies $c_3(q,m,k,p) = -c_3(q,k,m,p)$.\\
$c_4$ does not depend on $\lambda, \mu$.\\

\paragraph{Calculation of $\iota(\upsilon)$}
Suppose that $\hat W_r(k)$ is given by~\eqref{eq:W_hat_coefficients}. Rewriting this relation yields
\begin{align*}
	\hat W_r(k) = \frac{2}{\pi k}\sin(2\pi k r) = \frac{1}{k}f(kr),
\end{align*}
if $k\neq 0$ and $f\colon \R\to \R$ is given by $f(r) = \frac{2}{\pi}\sin(2\pi r)$. Moreover, note that
\begin{align*}
	\hat W_r(0) = k \hat W_\frac{r}{k}(0).
\end{align*}

We use this to calculate:
\begin{align*}
	H(q,r) = \frac{\hat W_r(0) + \hat W_r(2q) - 2\hat W_r(q)}{\hat W_r(q)} = u(qr),
\end{align*}
where $u\colon \R\to \R$ is defined by
\begin{align*}
	u(\upsilon) = \frac{4\upsilon + \frac 12 f(2\upsilon)-2f(\upsilon)}{f(\upsilon)}.
\end{align*}
Moreover,
\begin{align*}
	c_1(q,2q,p^0(0)) &= \frac 14 \left( \frac{1}{-q}f(-qr) + \frac{1}{3q}f(3qr) - \frac{2}{q}f(qr) - H(q,r)\frac 1q f(qr)\right)\\
	&=\frac 1q g(qr),
\end{align*}
with $g\colon \R\to \R$,
\begin{align*}
	g(\upsilon) = \frac 14\left( f(\upsilon) + \frac 13 f(3\upsilon) - 2f(\upsilon) - u(\upsilon)f(\upsilon)\right).
\end{align*}
Furthermore,
\begin{align*}
	c_3(q,2q,q,p^0(0)) &= \frac 18 \left( \frac{-1}{q}f(qr) + 8r - \frac 1q f(2qr) + \frac{1}{3q}f(3qr)\right)\\
	&=\frac{1}{q}h(qr),
\end{align*}
with $h\colon \R\to \R$,
\begin{align*}
	h(\upsilon) = \frac 18 \left(-f(\upsilon) + 8\upsilon - f(2\upsilon) + \frac 13 f(3\upsilon)\right).
\end{align*}
Finally, we can put everything together and calculate
\begin{align*}
	X(q,r) &= -\frac{1}{4q}f(qr) + \frac{h(qr)}{4g(qr)}\left(-4r + \frac{1}{2q}f(2qr)\right)\\
	&= \frac{1}{q}\iota(qr),
\end{align*}
with 
\begin{align*}
	\iota(\upsilon) = -\frac{1}{4}f(\upsilon) + \frac{h(\upsilon)}{4g(\upsilon)}\left(-4\upsilon + \frac 12 f(2\upsilon)\right).
\end{align*}
In conclusion, we have
\begin{align*}
	\iota(\upsilon) = \frac{-1}{2\pi}\sin(2\pi \upsilon)+\frac{-\sin(2\pi \upsilon) + 4\pi \upsilon - \sin(4\pi \upsilon) + \frac{1}{3}\sin(6\pi \upsilon)}{8\left(-\sin(2\pi \upsilon) + \frac 13 \sin(6\pi \upsilon) - (2 \pi \upsilon + \frac 12 \sin(4\pi \upsilon) - 2\sin(2\pi \upsilon))\right)}\left(-4\upsilon + \frac{1}{\pi}\sin(4\pi \upsilon)\right)
\end{align*}

\section{Supplementary Calculations}\label{sec:sup}
\subsection{Maximal Eigenvalue}

\begin{lemma}\label{lem:graph_eigvals}
For all $k\in \Z\setminus\{0\}$ and all $r$ with $0\le r \le \frac 12$ we have
\begin{align}\label{eq:graph_inequality}
    \hat W_r(1)\ge \hat W_r(k)
\end{align}
\end{lemma}

\begin{proof}
Because $\hat W_r(k) = \hat W_r(-k)$ this inequality only needs to be shown for $k\in \N, k\ge 2$. 
The derivative of $\hat W_r(k)$ with respect to $r$ is given by
\begin{align*}
    \frac{\d}{\d r}\hat W_r(k) = 4 \cos(2\pi k r).
\end{align*}
Since $\cos$ is decreasing in $[0,\pi/2]$ we get $\frac{\d}{\d r}\hat W_r(1) = 4\cos(2\pi r) \ge 4 \cos(2\pi r k) = \frac{\d}{\d r}\hat W_r(k)$ for all $r\in [0,\frac{1}{4k}]$. Since $\hat W_0(k) = 0$, equation~\eqref{eq:graph_inequality} follows for all $r\in [0,\frac{1}{4k}]$.
On the one hand, since this holds in particular for $r=\frac{1}{4k}$ and $\frac{\d}{\d r}\hat W_r(1)\ge 0$ for all $r\le \frac{1}{4}$ we have $\hat W_r(1) \ge \hat W_\frac{1}{4k}(1) \ge \hat W_\frac{1}{4k}(k) = \frac{2}{\pi k}$ for all $r\in [\frac{1}{4k}, \frac 14]$. On the other hand, $\hat W_r(k) \le \frac{2}{\pi k}$. Therefore the inequality~\eqref{eq:graph_inequality} extends to all $r\in [0, \frac{1}{4}]$. By a symmetry argument one can see that it even holds for all $r\in [0, \frac 12]$.
\end{proof}

\begin{theorem}\label{thm:grapheigenvalues}
    If $\frac{2}{\pi q} \le 2r-\frac{1}{\pi}\sin(2\pi r)$, the largest eigenvalue of the linearization of~\eqref{eq:pd_problem} with $\lambda = \mu = 0$ around a $q$-twisted state is attained for $k=q$, i.e., $\max_k c_1(q,k, (r,0,0)) = c_1(q,q,(r,0,0))$.
\end{theorem}

\begin{proof}
Let $k\in \N$, $k\neq q$ be fixed and $r\in [0,\frac 12]$, $q\in\N$ such that the assumption in the theorem is fulfilled. Then,
\begin{align*}
    \hat W_r(q+k)-\hat W_r(2q) & \le  \frac{2}{\pi}\left( \frac{1}{q+k}+\frac{1}{2q}\right)\\
    &\le \frac{4}{\pi q}\\
    &\le 4r-\frac{2}{\pi}\sin(2\pi r)\\
    &=\hat W_r(0)-\hat W_r(1).
\end{align*}
since $k-q\in \Z\setminus\{0\}$, we obtain by Lemma~\ref{lem:graph_eigvals}
\begin{align*}
    \hat W_r(q+k) - \hat W_r(2q) \le \hat W_r(0) - \hat W_r(k-q).
\end{align*}
After rearranging this inequality and subtracting $2\hat W_r(q)$ to both sides it reads as
\begin{align*}
    \hat W_r(2q) + \hat W_r(0) - 2\hat W_r(q) \ge \hat W_r(q+k) + \hat W_r(k-q) - 2\hat W_r(q).
\end{align*}
This is equivalent to $c_1(q,q,(r,0,0)) \ge c_1(q,k,(r,0,0))$ for all $k\neq q$. Thus, the proof is complete.
\end{proof}

\subsection{$\gamma$ Ratio in the Attractive Kuramoto Model}\label{sec:gamma_ratio}

Here, we consider the bifurcation in the attractive Kuramoto model studied in Section~\ref{sec:application_graph}. We show that $\gamma_1/\gamma_2>0$ when the $q$-twisted state looses its stability. To do this, we assume the following statements, which mainly follow from an analysis in~\cite{Wiley2006}:
\begin{itemize}
	\item At $r = r_0$ the first eigenvalue $c_1(q,1,(r,0,0))$ passes through $0$ from below when increasing $r$. Moreover, at the bifurcation the other eigenvalues are negative, i.e., $c_1(q,\ell, (r_0,0,0))<0$ for all $\ell\ge 2$.
	\item $c_1(q+1, 1, (r_0,0,0))>0$.
\end{itemize}

From the first statement, it immediately follows that $\gamma_2>0$. It remains to show that $\gamma_1>0$.

In particular, using the abbreviations from Appendix~\ref{sec:abbreviations} and $p_0 = (r_0,0,0)$, it follows from these statements that
\begin{align}
\label{eq:assumption1}
	c_1(q,1,p_0) &= \frac{1}{4}(\hat W_{r_0}(q-1) - 2 \hat W_{r_0}(q) + \hat W_{r_0}(q+1)) = 0\\
\label{eq:assumption2}
	c_1(q,2,p_0) &= \frac{1}{4}(\hat W_{r_0}(q-2) - 2 \hat W_{r_0}(q) + \hat W_{r_0}(q+2)) < 0\\
\label{eq:assumption3}
	c_1(q+1,1,p_0) &= \frac{1}{4}(\hat W_{r_0}(q) - 2\hat W_{r_0}(q+1) + \hat W_{r_0}(q+2)) > 0
\end{align}
Using~\eqref{eq:assumption1} we then get
\begin{align*}
	16c_5(q,1,p_0) &= \hat W_{r_0}(q-2) - 4 \hat W_{r_0}(q-1) + 6\hat W_{r_0}(q) - 4\hat W_{r_0}(q+1) + \hat W_{r_0}(q+2)\\
	&= \hat W_{r_0}(q-2)-2\hat W_{r_0}(q)+\hat W_{r_0}(q+2) = :q_1.
\end{align*}
Moreover, we note that
\begin{align*}
	c_2(q,1,p_0) = c_3(q,2,1,p_0) = -\hat W_{r_0}(q-2) +2\hat W_{r_0}(q-1)-2\hat W_{r_0}(q+1) + \hat W_{r_0}(q+2) =: q_2.
\end{align*}
Then, we obtain
\begin{align*}
	32\gamma_1 &= 16\left(c_5(q,1,p_0)-\frac{c_2(q,1,p_0)c_3(q,2,1,p_0)}{c_1(q,2,p_0)}\right)\\
	&=q_1 - \frac{q_2^2}{q_1}\\
	&= \frac{1}{q_1}(q_1+q_2)(q_1-q_2)
\end{align*}
By~\eqref{eq:assumption2} we have $q_1<0$. Furthermore, using~\eqref{eq:assumption1} and~\eqref{eq:assumption3} we obtain
\begin{align*}
	q_1+q_2 &= 2(\hat W_{r_0}(q-1)-\hat W_{r_0}(q)-\hat W_{r_0}(q+1) + \hat W_{r_0}(q+2))\\
	&=2(\hat W_{r_0}(q) - 2\hat W_{r_0}(q+1) + \hat W_{r_0}(q+2))\\
	&>0.
\end{align*}
Consequently, $q_2>0$. Thus, we conclude $q_1<0$, $q_1+q_2>0$ and $q_1-q_2<0$ and therefore $\gamma_1>0$.

\section{Derivative}\label{sec:derivative}

In this section we show that $F$ is indeed Fr\'echet differentiable. We give the operator that represents the Fr\'echet-Derivative but only show this without higher-order interactions, i.e., when $\lambda = \mu = 0$, since that does not complicate but only lengthen the calculations.

We claim, that the $n$-th Fr\'echet-Derivative of $F(\Psi,p)$ around a state $\tilde \Psi$ is given by a $n$-linear operator $A^{\tilde\Psi} \colon (H_0^1)^n\to H_0^1$ with
\begin{align*}
	&(A^{\tilde \Psi}[\eta_1,\dots,\eta_n])(x)\\
	&= \int_\S W_r(x-y)\prod_{i=1}^n(\eta_i(y)-\eta_i(x))\sin^{[n]}(\tilde \Psi(y)-\tilde \Psi(x))\ \d y - \int_\S W_r(y)\prod_{i=1}^n\eta_i(y)\sin^{[n]}(\tilde \Psi(y))\ \d y\\
	&\quad +\lambda \left[ \int_\S\int_\S W_r(z+y-2x) \prod_{i=1}^n(\eta_i(z)+\eta_i(y)-2\eta_i(x))\sin^{[n]}(\tilde \Psi(z)+\tilde \Psi(y) -2\tilde \Psi(x)) \ \d y \d z\right.\\
	&\qquad - \left. \int_\S\int_\S W_r(z+y)\prod_{i=1}^n (\eta_i(z)+\eta_i(y))\sin^{[n]}(\tilde \Psi(z)+\tilde \Psi(y))\ \d y \d z\right]\\
	&\quad +\mu\left[ \int_\S\int_\S\int_\S W_r(z-y+w-x)\prod_{i=1}^n(\eta_i(z)-\eta_i(y)+\eta_i(w)-\eta_i(x))\sin^{[n]}(\tilde\Psi(z)-\tilde\Psi(y)+\tilde\Psi(w)-\tilde \Psi(x)) \ \d w\d y \d z\right.\\
	&\qquad - \left. \int_\S\int_\S\int_\S W_r(z-y+w)\prod_{i=1}^n(\eta_i(z)-\eta_i(y)+\eta_i(w))\sin^{[n]}(\tilde \Psi(z)-\tilde \Psi(y)+\tilde \Psi(w))\ \d w \d y \d z\right],
\end{align*}
where $\sin^{[n]}$ denotes the $n$-th derivative of $\sin$. 

The main estimations needed to prove that this is indeed a $n$-linear operator and the $n$-th Fr\'echet derivative of $F$ are
\begin{align}\label{eq:graph_der_estimation}
	\abs{\eta(y)-\eta(x)}^2 &= \abs{\int_x^y \partial \eta(z)\ \d z}^2 \le \int_\S \partial \eta(z)^2\ \d z = \norm{\partial \eta}_{L^2}^2,
\end{align}
for the part without higher-order interactions and
\begin{align}
	\label{eq:lambda_der_estimation}
	\abs{\eta(z)+\eta(y)-2\eta(x)}^2 &\le 4\norm{\partial\eta}_{L^2}^2,\\
	\label{eq:mu_der_estimation}
	\abs{\eta(z)-\eta(y)+\eta(w)-\eta(x)}^2 &\le 4\norm{\partial\eta}_{L^2}^2
\end{align}
for the parts involving higher-order interactions.
Due to the similarity of the main estimations~\eqref{eq:graph_der_estimation},\eqref{eq:lambda_der_estimation} and~\eqref{eq:mu_der_estimation} regarding the parts with and without higher-order interactions, respectively, we only consider parts without higher-order interactions in the following. That means we only proof the boundedness of $A^{\tilde \Psi}$ and its derivative property for $\lambda = \mu = 0$.

\subsection{Boundedness of $A^{\tilde \Psi}$}

First we show that $A^{\tilde \Psi}$ is bounded.

Denote
\begin{align*}
f(x) = \int_\S W_r(x-y)\prod_{i=1}^n(\eta_i(y)-\eta_i(x))\sin^{[n]}(\tilde \Psi(y)-\tilde \Psi(x))\ \d y
\end{align*}
and
\begin{align*}
g(x) = \int_\S W_r(y)\prod_{i=1}^n\eta_i(y)\sin^{[n]}(\tilde \Psi(y))\ \d y.
\end{align*}
Then 
\begin{align*}
(A^{\tilde \Psi}[\eta_1,\dots,\eta_n])(x) = f(x) - g(x).
\end{align*}

Using $\norm{\sin^{[n]}}_\infty\le 1$ and $\norm{W_r}_\infty \le 1$ we estimate
\begin{align*}
\norm{f}_{L^2}^2 &= \int_\S \left( \int_\S W_r(x-y)\prod_{i=1}^n(\eta_i(y)-\eta_i(x))\sin^{[n]}(\tilde \Psi(y)-\tilde \Psi(x))\ \d y \right)^2 \d x \\
&\le \int_\S  \int_\S \prod_{i=1}^n(\eta_i(y)-\eta_i(x))^2\ \d y  \d x \\
& = \prod_{i=1}^n\norm{\partial \eta_i}_{L^2}^2.
\end{align*}
Moreover, a similar estimation yields 
\begin{align*}
	\norm{g}_{L^2}^2 \le \prod_{i=1}^n\norm{\eta_i}_{L^2}^2.
\end{align*}

Further, we calculate the derivative of $f$:
\begin{align*}
	&D_x f(x)\\
	 &= D_x\left( \int_{x-r}^{x+r} \prod_{i=1}^n (\eta_i(y)-\eta_i(x)) \sin^{[n]}(\tilde \Psi(y)-\tilde \Psi(x)) \ \d y \right)\\
	&=\prod_{i=1}^n(\eta_i(x+r)-\eta_i(x))\sin^{[n]}(\tilde \Psi(x+r)-\tilde \Psi(x)) - \prod_{i=1}^n(\eta_i(x-r)-\eta_i(x))\sin^{[n]}(\tilde \Psi(x-r)-\tilde \Psi(x))\\
	&\quad + \int_{x-r}^{x+r} \sum_{j=1}^n \prod_{\stackrel{i=1}{i\neq j}}^n(\eta_i(y)-\eta_i(x))(-\partial \eta_j(x))\sin^{[n]}(\tilde \Psi(y)-\tilde\Psi(x)) \  \d y\\
	&\quad + \int_{x-r}^{x+r} \prod_{i=1}^n(\eta_i(y)-\eta_i(x))\sin^{[n+1]}(\tilde \Psi(y)-\tilde \Psi(x))(-\partial \tilde \Psi(x)) \ \d y.
\end{align*}

Let us denote 
\begin{align*}
h_1(x) = \prod_{i=1}^n(\eta_i(x+r)-\eta_i(x))\sin^{[n]}(\tilde \Psi(x+r)-\tilde \Psi(x))\\
h_2(x) = - \prod_{i=1}^n(\eta_i(x-r)-\eta_i(x))\sin^{[n]}(\tilde \Psi(x-r)-\tilde \Psi(x))
\end{align*}
and
\begin{align*}
u_j(x) &= \int_{x-r}^{x+r} \prod_{\stackrel{i=1}{i\neq j}}^n(\eta_i(y)-\eta_i(x))(-\partial \eta_j(x))\sin^{[n]}(\tilde \Psi(y)-\tilde\Psi(x))\ \d y,\\
q(x) &= \int_{x-r}^{x+r} \prod_{i=1}^n(\eta_i(y)-\eta_i(x))\sin^{[n+1]}(\tilde \Psi(y)-\tilde \Psi(x))(-\partial \tilde \Psi(x)) \ \d y
\end{align*}
Then, $D_x(A^{\tilde \Psi}[\eta_1,\dots,\eta_n])(x) = D_xf(x) = h_1(x) + h_2(x) + \sum u_j(x) + q(x)$. An estimation shows that
\begin{align*}
	\norm{h_1}_{L^2}^2 &= \int_\S \left[ \prod_{i=1}^n(\eta_i(x+r)-\eta_i(x))\sin^{[n]}(\tilde \Psi(x+r)-\tilde \Psi(x))\right]^2\ \d x\\
	&\le \int_\S \prod_{i=1}^n(\eta_i(x+r)-\eta_i(x))^2\ \d x\\
	&\le \prod_{i=1}^n \norm{\partial\eta_i}_{L^2}^2.
\end{align*}
Similarly, we obtain $\norm{h_2}_{L^2}^2 \le \prod_{i=1}^n \norm{\partial\eta_i}_{L^2}^2$. Moreover,
\begin{align*}
	\norm{u_j}_{L^2}^2 &= \int_\S \left( \int_{x-r}^{x+r}  \prod_{\stackrel{i=1}{i\neq j}}^n(\eta_i(y)-\eta_i(x))(-\partial \eta_j(x))\sin^{[n]}(\tilde \Psi(y)-\tilde\Psi(x))\ \d y \right)^2\ \d x\\
	&\le \int_\S \int_\S \prod_{\stackrel{i=1}{i\neq j}}^n \norm{\partial \eta_i}_{L^2}^2 (\partial\eta_j(x))^2 \ \d y \ \d x\\
	& = \prod_{i=1}^n \norm{\partial\eta_i}_{L^2}^2
\end{align*}
and
\begin{align*}
	\norm{q}_{L^2}^2 &= \int_\S \left( \int_{x-r}^{x+r} \prod_{i=1}^n(\eta_i(y)-\eta_i(x))\sin^{[n+1]}(\tilde \Psi(y)-\tilde \Psi(x))(-\partial \tilde \Psi(x)) \ \d y\right)^2 \ \d x\\
	&\le \int_\S \int_\S \prod_{i=1}^n \norm{\partial\eta_i}_{L^2}^2 (\partial \tilde \Psi(x))^2\ \d y \ \d x\\
	&\le  \prod_{i=1}^n \norm{\partial\eta_i}_{L^2}^2 \norm{\partial \tilde \Psi}_{L^2}^2
\end{align*}

All together, we obtain:
\begin{align*}
	\norm{A^{\tilde \Psi}[\eta_1,\dots,\eta_n]}_{H^1}^2 &= \norm{A^{\tilde \Psi}[\eta_1,\dots,\eta_n]}_{L^2}^2 + \norm{D_x(A^{\tilde \Psi}[\eta_1,\dots,\eta_n])}_{L^2}^2\\
	&=\norm{f-g}_{L^2}^2 + \norm{h_1+h_2 + \sum_{j=1}^n u_j + q}_{L^2}^2\\
	&\le 2\norm f_{L^2}^2 + 2\norm{g}_{L^2}^2 + (n+3)\left( \norm{h_1}_{L^2}^2 +\norm{h_2}_{L^2}^2 + \sum_{j=1}^n \norm{u_j}_{L^2}^2 + \norm{q}_{L^2}^2\right)\\
	&\le 2 \prod_{i=1}^n\norm{\partial \eta_i}_{L^2}^2 + 2\prod_{i=1}^n \norm{\eta_i}^2_{L^2} + (n+3)\left( (n+2)\prod_{i=1}^n\norm{\partial \eta_i}_{L^2}^2 + \prod_{i=1}^n\norm{\partial \eta_i}_{L^2}^2        \norm{\partial \tilde \Psi}_{L^2}^2 \right)\\
	&\le 4 \prod_{i=1}^n\norm{\eta_i}_{H^1}^2 + (n+3)\left( (n+2)\prod_{i=1}^n\norm{\eta_i}_{H^1}^2 + \prod_{i=1}^n\norm{\eta_i}_{H^1}^2 \norm{\tilde \Psi}_{H^1}^2 \right)\\
	&\le c \prod_{i=1}^n\norm{\eta_i}_{H^1}^2,
\end{align*}
where $c$ can be chosen as
\begin{align*}
	c = 4+(n+3)\left(n+2+\norm{\tilde \Psi}_{H^1}^2\right).
\end{align*}
Even though some estimations are far from being tight, this proves that $A^{\tilde \Psi}[\eta_1,\dots,\eta_n]$ is a bounded $n$-linear operator.

\subsection{Derivative Property}

To inductively show that $A^{\tilde \Psi}[\eta_1,\dots,\eta_n]$ is the $n$-th Fr\'echet derivative of $F$ around $\Tilde \Psi$ we need to confirm that
\begin{align}\label{eq:Frechet}
	\lim_{\norm{\eta_n}_{H^1}\to 0} \frac{1}{\norm{\eta_n}_{H^1}} \norm{A^{\tilde \Psi + \eta_n}[\eta_1,\dots,\eta_{n-1}] - A^{\tilde \Psi}[\eta_1,\dots,\eta_{n-1}] - A^{\tilde \Psi}[\eta_1,\dots,\eta_n]}_{H^1} = 0.
\end{align}

First we rewrite \allowdisplaybreaks
\begin{align}
\nonumber &A^{\tilde \Psi + \eta_n}[\eta_1,\dots,\eta_{n-1}] - A^{\tilde \Psi}[\eta_1,\dots,\eta_{n-1}] - A^{\tilde \Psi}[\eta_1,\dots,\eta_n]\\
\nonumber &=\int_\S W_r(x-y)\prod_{i=1}^{n-1}(\eta_i(y)-\eta_i(x))\sin^{[n-1]}(\tilde \Psi(y)-\tilde \Psi(x)+\eta_n(y)-\eta_n(x))\ \d y\\
\nonumber &\quad - \int_\S W_r(y)\prod_{i=1}^{n-1}\eta_i(y)\sin^{[n-1]}(\tilde \Psi(y)+\eta_n(y))\ \d y\\
\nonumber &\quad - \int_\S W_r(x-y)\prod_{i=1}^{n-1}(\eta_i(y)-\eta_i(x))\sin^{[n-1]}(\tilde \Psi(y)-\tilde \Psi(x))\ \d y\\
\nonumber &\quad + \int_\S W_r(y)\prod_{i=1}^{n-1}\eta_i(y)\sin^{[n-1]}(\tilde \Psi(y))\ \d y\\
\nonumber &\quad - \int_\S W_r(x-y)\prod_{i=1}^n(\eta_i(y)-\eta_i(x))\sin^{[n]}(\tilde \Psi(y)-\tilde \Psi(x))\ \d y\\
\nonumber &\quad +  \int_\S W_r(y)\prod_{i=1}^n\eta_i(y)\sin^{[n]}(\tilde \Psi(y))\ \d y\\
\begin{split}\label{eq:abb_g1}
&=\int_\S  W_r(x-y) \prod_{i=1}^{n-1}(\eta_i(y)-\eta_i(x)) \left[ \sin^{[n-1]}(\tilde \Psi(y)-\tilde \Psi(x)+\eta_n(y)-\eta_n(x))\right.\\
&\quad \left.-\sin^{[n-1]}(\tilde \Psi(y)-\tilde \Psi(x))-(\eta_n(y)-\eta_n(x))\sin^{[n]}(\tilde \Psi(y)-\tilde \Psi(x))\right]\ \d y
\end{split}\\
\label{eq:abb_g2}
&\quad - \int_\S W_r(y)\prod_{i=1}^{n-1}\eta_i(y) \left[ \sin^{[n-1]}(\tilde \Psi(y)+\eta_n(y))-\sin^{[n-1]}(\tilde \Psi(y))-\eta_n(y)\sin^{[n]}(\tilde \Psi(y))\right] \ \d y
\end{align}

By introducing the notation $g_1(x)$ for~\eqref{eq:abb_g1} and $g_2(x)$ for~\eqref{eq:abb_g2} we recover
\begin{align*}
A^{\tilde \Psi + \eta_n}[\eta_1,\dots,\eta_{n-1}] - A^{\tilde \Psi}[\eta_1,\dots,\eta_{n-1}] - A^{\tilde \Psi}[\eta_1,\dots,\eta_n] = g_1(x) + g_2(x).
\end{align*}
Note that by Taylor's Theorem we have
\begin{align*}
	f(x_0+a) = f(x_0) + af'(x_0) + \frac{f''(\xi)}{2} a^2
\end{align*}
for each twice continuously differentiable function $f$ and some $\xi\in (x_0, x_0+a)$ if $a>0$ and $\xi\in (x-0-a,x_0)$ if $a<0$. By applying this theorem to $f=\sin^{[n-1]}$ it follows that
\begin{align}\label{eq:taylor_sin}
	|\sin^{[n-1]}(x_0+a) - \sin^{[n-1]}(x_0) - a\sin^{[n]}(x_0)| \le \frac{a^2}{2}.
\end{align}
Using this inequality we can estimate
\begin{align*}
\norm{g_1}_{L^2}^2 &= \int_\S \left\{ \int_\S  W_r(x-y) \prod_{i=1}^{n-1}(\eta_i(y)-\eta_i(x)) \left[ \sin^{[n-1]}(\tilde \Psi(y)-\tilde \Psi(x)+\eta_n(y)-\eta_n(x))\right.\right.\\
&\quad \left.\left.-\sin^{[n-1]}(\tilde \Psi(y)-\tilde \Psi(x))-(\eta_n(y)-\eta_n(x))\sin^{[n]}(\tilde \Psi(y)-\tilde \Psi(x))\right]\ \d y \right\}^2 \ \d x\\
&\le \prod_{i=1}^{n-1}\norm{\partial \eta_i}_{L^2}^2 \int_\S\left( \int_\S \frac 12 (\eta_n(y)-\eta_n(x))^2\ \d y\right)^2 \d x\\
&=\frac 14  \prod_{i=1}^{n-1}\norm{\partial \eta_i}_{L^2}^2 \norm{\partial\eta_n}_{L^2}^4.
\end{align*}
Furthermore, also by using~\eqref{eq:taylor_sin}, we obtain
\begin{align*}
	\norm{g_2}_{L^2}^2 &= \int_\S \left( \int_\S W_r(y)\prod_{i=1}^{n-1}\eta_i(y) \left[ \sin^{[n-1]}(\tilde \Psi(y)+\eta_n(y))-\sin^{[n-1]}(\tilde \Psi(y))-\eta_n(y)\sin^{[n]}(\tilde \Psi(y))\right] \ \d y \right)^2\ \d x\\
	&\le \prod_{i=1}^{n-1}\norm{\eta_i}_{L^2}^2 \int_\S \left( \int_\S \frac 12 \abs{\eta_n(y)}^2\ \d y \right)^2\d x\\
	& = \frac 14 \prod_{i=1}^{n-1}\norm{\eta_i}_{L^2}^2 \norm{\eta_n}_{L^2}^4.
\end{align*}

Now, we calculate the derivative
\begin{align}
	\nonumber
	&D_x(A^{\tilde \Psi + \eta_n}[\eta_1,\dots,\eta_{n-1}] - A^{\tilde \Psi}[\eta_1,\dots,\eta_{n-1}] - A^{\tilde \Psi}[\eta_1,\dots,\eta_n])\\
	\nonumber
	&\quad = D_xg_1(x)\\
	\nonumber
	&\quad = D_x\left( \int_{x-r}^{x+r}\prod_{i=1}^{n-1}(\eta_i(y)-\eta_i(x)) \left[ \sin^{[n-1]}(\tilde \Psi(y)-\tilde\Psi(x) + \eta_n(y)-\eta_n(x)) \right.\right.\\
	\nonumber
	&\qquad \left.\left.-\sin^{[n-1]}(\tilde \Psi(y)-\tilde \Psi(x))-(\eta_n(y)-\eta_n(x))\sin^{[n]}(\tilde \Psi(y)-\tilde \Psi(x))\right] \ \d y \right)\\
	\begin{split}\label{eq:abb_u1}
	&\quad = \prod_{i=1}^{n-1}(\eta_i(x+r)-\eta_i(x)) \left[ \sin^{[n-1]}(\tilde \Psi(x+r)-\tilde\Psi(x) + \eta_n(x+r)-\eta_n(x)) \right.\\
	&\qquad \quad \left.-\sin^{[n-1]}(\tilde \Psi(x+r)-\tilde \Psi(x))-(\eta_n(x+r)-\eta_n(x))\sin^{[n]}(\tilde \Psi(x+r)-\tilde \Psi(x))\right]
	\end{split}\\
	\begin{split}\label{eq:abb_u2}
	&\qquad - \prod_{i=1}^{n-1}(\eta_i(x-r)-\eta_i(x)) \left[ \sin^{[n-1]}(\tilde \Psi(x-r)-\tilde\Psi(x) + \eta_n(x-r)-\eta_n(x)) \right.\\
	&\qquad \quad \left.-\sin^{[n-1]}(\tilde \Psi(x-r)-\tilde \Psi(x))-(\eta_n(x-r)-\eta_n(x))\sin^{[n]}(\tilde \Psi(x-r)-\tilde \Psi(x))\right]
	\end{split}\\
	\begin{split}\label{eq:abb_u3}
	&\qquad +\sum_{j=1}^{n-1} \int_{x-r}^{x+r} \prod_{\stackrel{i=1}{i\neq j}}^{n-1}(\eta_i(y)-\eta_i(x)) (-\partial\eta_j(x)) \left[ \sin^{[n-1]}(\tilde \Psi(y)-\tilde\Psi(x) + \eta_n(y)-\eta_n(x)) \right.\\
	&\qquad \quad \left.-\sin^{[n-1]}(\tilde \Psi(y)-\tilde \Psi(x))-(\eta_n(y)-\eta_n(x))\sin^{[n]}(\tilde \Psi(y)-\tilde \Psi(x))\right]\ \d y
	\end{split}\\
	\begin{split}\label{eq:abb_u4}
	&\qquad + \int_{x-r}^{x+r}\prod_{i=1}^{n-1}(\eta_i(y)-\eta_i(x))\left[ \sin^{[n]}(\tilde \Psi(y)-\tilde\Psi(x) + \eta_n(y)-\eta_n(x))(-\partial \tilde\Psi(x)-\partial\eta_n(x))\right.\\
	&\qquad \quad \left.- \sin^{[n]}(\tilde \Psi(y)-\tilde \Psi(x))(-\partial\tilde \Psi(x))-(-\partial \eta_n(x))\sin^{[n]}(\tilde \Psi(y)-\tilde \Psi(x))\right.\\
	&\qquad \quad \left. - (\eta_n(y)-\eta_n(x))\sin^{[n+1]}(\tilde \Psi(y)-\tilde \Psi(x))(-\partial \tilde \Psi(x)) \right]\ \d y
	\end{split}
\end{align}

We use the abbreviations $u_1(x)$ for~\eqref{eq:abb_u1}, $u_2(x)$ for~\eqref{eq:abb_u2}, $u_3^j(x)$ for the $j$-th summand in~\eqref{eq:abb_u3} and $u_4(x)$ for~\eqref{eq:abb_u4} such that $D_x(A^{\tilde \Psi + \eta_n}[\eta_1,\dots,\eta_{n-1}] - A^{\tilde \Psi}[\eta_1,\dots,\eta_{n-1}] - A^{\tilde \Psi}[\eta_1,\dots,\eta_n]) = u_1(x) + u_2(x) + \sum_j u_3^j(x) + u_4(x)$. Further, we split $u_4$ into $u_{4,1}(x)+u_{4,2}(x)$ as
\begin{align*}
	u_{4,1}(x) &= \int_{x-r}^{x+r}\prod_{i=1}^{n-1}(\eta_i(y)-\eta_i(x))(-\partial \tilde\Psi(x))\left[ \sin^{[n]}(\tilde \Psi(y)-\tilde\Psi(x) + \eta_n(y)-\eta_n(x))\right.\\
	&\quad \left.- \sin^{[n]}(\tilde \Psi(y)-\tilde \Psi(x)) - (\eta_n(y)-\eta_n(x))\sin^{[n+1]}(\tilde \Psi(y)-\tilde \Psi(x)) \right]\ \d y\\
	u_{4,2}(x) &=  \int_{x-r}^{x+r}\prod_{i=1}^{n-1}(\eta_i(y)-\eta_i(x))(-\partial\eta_n(x))\left[ \sin^{[n]}(\tilde \Psi(y)-\tilde\Psi(x) + \eta_n(y)-\eta_n(x))\right.\\
	&\qquad \quad \left.-\sin^{[n]}(\tilde \Psi(y)-\tilde \Psi(x)) \right]\ \d y
\end{align*}

Then, again by using Taylor's theorem, we can estimate
\begin{align*}
	\norm{u_1}_{L^2}^2 &\le \prod_{i=1}^{n-1} \norm{\partial \eta_i}_{L^2}^2 \int_\S \left(\sin^{[n-1]}(\tilde \Psi(x+r)-\tilde\Psi(x) + \eta_n(x+r)-\eta_n(x)) \right.\\
	&\quad \left.-\sin^{[n-1]}(\tilde \Psi(x+r)-\tilde \Psi(x))-(\eta_n(x+r)-\eta_n(x))\sin^{[n]}(\tilde \Psi(x+r)-\tilde \Psi(x))\right)^2 \ \d x\\
	&\le \prod_{i=1}^{n-1} \norm{\partial \eta_i}_{L^2}^2 \int_\S \left( \frac 12 (\eta_n(x+r)-\eta_n(x))^2\right)^2 \ \d x\\
	&\le \frac 14\prod_{i=1}^{n-1} \norm{\partial \eta_i}_{L^2}^2 \norm{\partial\eta_n}_{L^2}^4.
\end{align*}
Similarly, $\norm{u_2}_{L^2}^2 \le \frac 14\prod_{i=1}^{n-1} \norm{\partial \eta_i}_{L^2}^2 \norm{\partial\eta_n}_{L^2}^4$. Moreover,
\begin{align*}
	\norm{u_3^j}_{L^2}^2 &\le \prod_{\stackrel{i=1}{i\neq j}}^{n-1} \norm{\partial\eta_i}_{L^2}^2 \int_\S\int_\S (\partial \eta_j(x))^2 \left[ \sin^{[n-1]}(\tilde \Psi(y)-\tilde\Psi(x) + \eta_n(y)-\eta_n(x)) \right.\\
	&\quad \left.-\sin^{[n-1]}(\tilde \Psi(y)-\tilde \Psi(x))-(\eta_n(y)-\eta_n(x))\sin^{[n]}(\tilde \Psi(y)-\tilde \Psi(x))\right]^2\ \d y \d x\\
	&\le \prod_{\stackrel{i=1}{i\neq j}}^{n-1} \norm{\partial\eta_i}_{L^2}^2 \int_\S\int_\S (\partial \eta_j(x))^2 \left[ \frac 12 (\eta_n(y)-\eta_n(x))^2\right]^2\ \d y \d x\\
	&\le \frac 14\prod_{\stackrel{i=1}{i\neq j}}^{n-1} (\norm{\partial\eta_i}_{L^2}^2) \norm{\partial\eta_n}_{L^2}^4 \norm{\partial \eta_j}_{L^2}^2\\
	&= \frac 14 \prod_{i=1}^{n-1}( \norm{\partial \eta_i}_{L^2}^2 )\norm{\partial\eta_n}_{L^2}^4
\end{align*}
and
\begin{align*}
	\norm{u_{4,1}}_{L^2}^2 & \le \prod_{i=1}^{n-1} \norm{\partial\eta_i}_{L^2}^2 \int_\S \int_\S\left( (-\partial\tilde \Psi(x)) \frac 12 (\eta_n(y)-\eta_n(x))^2 \right)^2\ \d y \d x\\
	&\le \frac 14\prod_{i=1}^{n-1} \norm{\partial\eta_i}_{L^2}^2 \norm{\partial\eta_n}_{L^2}^4 \norm{\partial \tilde \Psi}_{L^2}^2,\\
	\norm{u_{4,2}}_{L^2}^2 &\le \prod_{i=1}^{n-1} \norm{\partial\eta_i}_{L^2}^2 \int_\S \int_\S \left[ (-\partial \eta_n(x)) (\eta_n(y)-\eta_n(x))\right]^2\ \d y\d x\\
	&\le \prod_{i=1}^{n-1} \norm{\partial\eta_i}_{L^2}^2\norm{\partial\eta_n}_{L^2}^4
\end{align*}
Finally, we can combine the estimations to obtain
\begin{align*}
&\norm{A^{\tilde \Psi + \eta_n}[\eta_1,\dots,\eta_{n-1}] - A^{\tilde \Psi}[\eta_1,\dots,\eta_{n-1}] - A^{\tilde \Psi}[\eta_1,\dots,\eta_n]}_{H^1}^2\\
&\le \norm{A^{\tilde \Psi + \eta_n}[\eta_1,\dots,\eta_{n-1}] - A^{\tilde \Psi}[\eta_1,\dots,\eta_{n-1}] - A^{\tilde \Psi}[\eta_1,\dots,\eta_n]}_{L^2}^2\\
&\quad + \norm{D_x(A^{\tilde \Psi + \eta_n}[\eta_1,\dots,\eta_{n-1}] - A^{\tilde \Psi}[\eta_1,\dots,\eta_{n-1}] - A^{\tilde \Psi}[\eta_1,\dots,\eta_n])}_{L^2}^2\\
&\le \norm{g_1 + g_2}_{L^2}^2 + \norm{u_1 + u_2 + \sum_{j=1}^{n-1} u_3^j + u_{4,1} + u_{4,2}}_{L^2}^2\\
&\le 2(\norm{g_1}_{L^2}^2 + \norm{g_2}_{L^2}^2) + (n+3)\left(\norm{u_1}_{L^2}^2+\norm{u_2}_{L^2}^2 + \sum_{j=1}^{n-1} \norm{u_3^j}_{L^2}^2 + \norm{u_{4,1}}_{L^2}^2 + \norm{u_{4,2}}_{L^2}^2\right)\\
&\le \prod_{i=1}^{n-1}\norm{\eta_i}_{H^1}^2 \norm{\eta_n}_{H^1}^4 + (n+3)\left( \frac{n+6}{4}\prod_{i=1}^{n-1}\norm{\eta_i}_{H^1}^2 \norm{\eta_n}_{H^1}^4 +  \frac 14 \prod_{i=1}^{n-1}\norm{\eta_i}_{H^1}^2 \norm{\eta_n}_{H^1}^4 \norm{\tilde \Psi}_{H^1}  \right) \\
&\le c \norm{\eta_n}_{H^1}^4,
\end{align*}
where $c$ can be chosen to be
\begin{align*}
	c = \prod_{i=1}^{n-1}\norm{\eta_i}_{H^1}^2\left(1+(n+3)\left(\frac{n+6}{4}+\frac 14 \norm{\tilde \Psi}_{H^1} \right)\right)
\end{align*}
This confirms~\eqref{eq:Frechet} and therefore concludes this section.

\bibliographystyle{plain}
\bibliography{main}

\end{document}